\documentclass[12pt]{amsart}
\usepackage[headings]{fullpage}

\usepackage{amsmath,amssymb,amsthm}
\usepackage{graphicx}
\usepackage{enumerate}
\usepackage{url}
\usepackage{slashed}
\usepackage{bm}
\RequirePackage{lineno}

\usepackage[bookmarks=true,%
    colorlinks=true,%
    linkcolor=blue,%
    citecolor=blue,%
    filecolor=blue,%
    menucolor=blue,%
    urlcolor=blue,%
    breaklinks=true]{hyperref}

\usepackage[all]{xy}
\usepackage{verbatim}
\usepackage{tikz-cd}

\newtheorem{thm}{Theorem}[]
\newtheorem*{thm*}{Theorem}
\newtheorem{lem}[thm]{Lemma}
\newtheorem{prop}[thm]{Proposition}

\newtheorem{ex}[thm]{Example}
\newtheorem{rem}[thm]{Remark}
\newtheorem{defn}[thm]{Definition}

\newcommand{\RR}{{\mathbb{R}}}
\newcommand{\Z}{{\mathbb{Z}}}

\def\mcg{\mathrm{Mod}(\Sigma)}
\def\Vol{\mathrm{Vol}}

\def\be{\begin{equation}}
\def\ee{\end{equation}}

\newcommand{\param}{{\mathchoice{\mkern1mu\mbox{\raise2.2pt\hbox{$
\centerdot$}}
\mkern1mu}{\mkern1mu\mbox{\raise2.2pt\hbox{$\centerdot$}}\mkern1mu}{
\mkern1.5mu\centerdot\mkern1.5mu}{\mkern1.5mu\centerdot\mkern1.5mu}}}
\begin{document}
%\setpagewiselinenumbers
%\linenumbers

\title[The twist-cofinite topology]{
The twist-cofinite topology on the mapping class group of a surface} 
\author{Ingrid Irmer}
\address{International Center for Mathematics\\
Southern University of Science and Technology\\Shenzhen, China
}
\email{ingridmary@sustech.edu.cn}
\today
%\thanks{
% {\em Mathematics Classification.} Primary 57N10. Secondary 57M07.
%\newline

%There is an interesting discussion here https://mathoverflow.net/questions/179269/is-a-generic-closed-orientable-hyperbolic-3-manifold-haken
%Can I show that a generic 3-manifold is Haken?
%The bibliography needs work

\begin{abstract}
A topology is defined on the mapping class group $\mcg$ of a compact, connected, orientable surface $\Sigma$. It is shown that  a notion of ``genericity" on subsets of $\mcg$ arises from this definition. Many plausible results follow from this notion easily; for example, the set of pseudo-Anosov maps is shown to be generic, and can be assumed to have arbitrary large stretch factor, generically. Let $M$ be a 3-manifold obtained from a Heegaard splitting of fixed genus $g$ and generic gluing map.  It is shown that for such manifolds, generically $b_1(M)=0,$ $M$ is hyperbolic and $M$ has Heegaard genus exactly $g.$
\end{abstract}

\maketitle

{\footnotesize
\tableofcontents
}
%%%%%%%%%%%%%%%%%%%%%%%%%%%%%%%%%%%%%%%%%%%%%%%%%%%%%%%%%%%%%%%%%%%%%%%%%%%%%%%%%%%%%%%%%%%%%%%%%%%%%%%%%%%%%%%%%%%%%%%%%%%%%%%
%%%%%%%%%%%%%%%%%%%%%%%%%%%%%%%%%%%%%%%%%%%%%%%%%%%%%%%%%%%%%%%%%%%%%%%%%%%%%%%%%%%%%%%%%%%%%%%%%%%%%%%%%%%%%%%%%%%%%%%%%%%%%%%

\section{Introduction}
\label{sec.intro}
Let $\Sigma$ be a closed, compact, connected, orientable surface of genus $g\geq 2.$ The mapping class group of $\Sigma$ is the group $\mcg=\pi_0(\mathrm{Diff}^+(\Sigma))$ of isotopy classes of positively oriented diffeomorphisms of $\Sigma.$ For any property $\mathcal{P}(f)$ depending on a mapping class $f \in \mcg,$ one may ask whether the property $\mathcal{P}(f)$ is true for a generic element of $\mcg.$\\

Such questions have been investigated by various authors, for example, \cite{DT06}, \cite{R14a}, \cite{R14b} and \cite{LMW14}. Although there are many possible notions of genericity of a mapping class group element, in most of the above works, genericity is defined using the following measure theoretic definition: Choose a finite generating set $S$ of $\mcg,$ and let $W_n$ be a random variable in $\mcg$ which is obtained by an $n$-step random walk in the Cayley graph of $\mcg,$ starting at the identity. The property $\mathcal{P}$ is said to be generic if, in the limit as $n\rightarrow \infty$, the probability that  $\mathcal{P}$ is true for an element $W_{n}$ approaches 1. In the papers \cite{R14a} and \cite{EST}, several versions of the notion of random elements of $\mcg$ are developed, but these notions of genericity stay probabilistic in nature.\\
%$$\mathbb{P}_{n \rightarrow \infty}(\mathcal{P}(W_n) \ \textrm{is true}) \rightarrow 1.$$

For subsets of $\RR,$ there are two alternative notions of being generic:
\begin{itemize}
\item[-]A measure theoretic one, stating that a subset $S \subset \RR$ is generic (in the sense of Lebesgue) if $\RR \setminus S$ has Lebesgue measure $0.$
\item[-]A topological one, stating that $S$ is generic (in the sense of Baire) if $S$ contains a countable intersection of open dense subsets of $\RR.$
\end{itemize}
The aim of this paper is to develop a topological notion of genericity of a subset of the mapping class group. In order to achieve this goal, in \cite{Detcherry}, a topology on $\mcg$ called the \textit{twist-cofinite topology} was defined. For $\gamma$ a non-trivial simple closed curve on $\Sigma,$ let $D_{\gamma}$ be the Dehn-twist around $\gamma.$

\begin{defn}[Twist-cofinite Topology]
A set of open sets is given by $$\mathcal{O}=\lbrace U \subset \mcg \ | \ \forall f \in U, \forall \gamma, f \circ D_{\gamma}^n \in U \ \textrm{for all except at most finitely many } n \in \Z  \rbrace.$$
\end{defn}

\begin{thm}\label{thm:topology}
The open sets $\mathcal{O}$ define a topology on $\mcg,$ and $\mcg$ is a connected topological group for this topology.
\end{thm}

A fundamental property of the twist-cofinite topology, similar for example to the Zariski topology, is that open sets are quite large. Specifically, a subset $S$ of $\mcg$ will be called \textit{thick} if $S$ has non-empty interior. The next proposition suggests that thickness of a subset is a good notion of genericity:

\begin{prop}\label{prop:topology}In the topology $\mathcal{O},$ 
\begin{itemize}
\item[-]The intersection $U_1 \cap U_2$ of two non-empty open sets $U_1$ and $U_2$ is non-empty.
\item[-]Any thick subset is dense.
\item[-]The intersection  $S_{1} \cap S_{2}$ of  two thick subsets is thick.
\end{itemize}
\end{prop}

With this notion of genericity, generic properties of subsets of $\mcg$ will be studied. For example,

\begin{thm}\label{thm:pseudoAnosov}The following subsets of $\mcg$ are non-empty and open (i.e. thick):
\begin{itemize}
\item[-]The set of pseudo-Anosov maps $f \in \mcg.$
\item[-]The sets $\lbrace f \in \mcg \ \textrm{pseudo-Anosov} \ | \ \Vol (M_{f})>C \rbrace,$ where $M_f$ is the mapping torus of $f$ and $C>0.$
\item[-]The sets $\lbrace f \in \mcg \ \textrm{pseudo-Anosov} \ | \ \lambda(f)>C \rbrace,$ where $\lambda(f)$ is the stretch factor of $f$ and $C>0.$
\item[-]The set of pseudo-Anosovs that do not leave any integer homology class of curves on $\Sigma$ invariant.
\end{itemize}
\end{thm}

In \cite{DT06} Dunfield and Thurston used the measure theoretic notion of genericity in $\mcg$ and Heegaard splittings to study the genericity of various properties of $3$-manifolds. The same notion of random $3$-manifolds was studied in subsequent papers \cite{R14a} \cite{R14b} and \cite{LMW14} by various authors.\\

Similar results can be obtained with the notion of thickness/genericity  defined above.\\

Let $H_{A}$ and $H_{B}$ be a pair of handlebodies, with $\partial H_{A}=\partial H_{B}=\Sigma$.  For $f \in \mcg,$ let $N_f$ be the $3$-manifold:
$$N_f = H_{A} \underset{f}{\cup} H_{B}$$ where $f$ is called the gluing map. The properties of $3$-manifold invariants of $N_f$ for a generic $f$ can be studied. Some of the simplest such invariants are first Betti numbers. Denote by $b_{1}(N_{f},\mathbb{Z})$ the first Betti number of $N_{f}$,  and  by $|H_{1}(N_{f};\Z)|$ the number of elements in the group $H_{1}(N_{f};\Z)$.

\begin{thm}
\label{thm:Betti} The following are all thick:
\begin{enumerate}
\item The set $\lbrace f \in \mcg \ | \ b_{1}(N_{f},\mathbb{Z})=0\rbrace$ 
\item The set $\lbrace f \in \mcg \ | \ |H_{1}(N_{f};\Z)|\geqslant k \rbrace$ for any integer $k\geqslant 0$
\end{enumerate}
Suppose $p$ is prime and $0\leqslant i \leqslant g$. In contrast, the set $\lbrace f \in \mcg \ | \ b_{1}(N_{f},\mathbb{F}_p)=i\rbrace$ is dense but not thick.
\end{thm}

The following will also be proven, using somewhat different techniques:

\begin{thm}
\label{thm:stabilisation}
 The sets $\lbrace f \in \mcg \ | \ g(N_f)=g(\Sigma) \rbrace$ and $\lbrace f \in \mcg \ | \ g(N_{f})\geq n\rbrace$  for every $n\geq 2$ are thick, where  $g(N_{f})$ is the Heegaard genus of $N_{f}$.
\end{thm}

A natural question to ask is whether, with this notion of genericity, a generic Heegaard splitting gives an irreducible manifold and whether it gives a hyperbolic manifold. Note that due to the Geometrisation theorem,  being hyperbolic is equivalent to being irreducible and atoroidal. A more general question is therefore whether $N_f$ admits an embedded incompressible surface of genus less than or equal to $ g$.

\begin{thm}
\label{thm:incompressible}
For any $g\geqslant 0,$ the set 
$$\lbrace f \in \mcg \ | \ N_f \ \textrm{has no embedded incompressible surface of genus }\leqslant g\rbrace$$
is thick.
\end{thm}

This theorem will be reduced to studying a property of Hempel distance in Harvey's complex of curves.  Hempel distance also gives a lower bound on the genus of an incompressible surface. Theorem \ref{thm:incompressible} will be shown to follow from Theorem \ref{thm:Hempel} below:

\begin{thm}\label{thm:Hempel}
For any $d\geqslant 3,$ the set 
$$\lbrace f \in \mcg \ | \ \textrm{the Heegaard splitting induced by }f \ \textrm{has distance}\geqslant d \rbrace$$
is an open dense set.
\end{thm}

While most of the results in this paper follow almost immediately from the definitions or from classical results, the proof of a lemma needed for these last two theorems is quite technical.\\

On page 3 of \cite{R14b}, a list of properties of elements of subgroups of mapping class groups represented by random long words in a symmetric generating set is given. On page 4 of the same paper, a list of properties of random Heegaard splittings according to the Dunfield-Thurston model is given. The properties investigated for genericity in this paper are based on these two lists. Results in this paper should be compared with theorems in \cite{R14b} and the extensive list of references given therein. Given the amount of literature available on the subject, the author apologises that no attempt is made here to do justice to related results in different frameworks.\\

As a final comment, the notion of genericity developed here has advantages and disadvantages. The measure theoretic definition of genericity a priori depends on the choice of a finite generating set of $\mcg.$ Genericity of a property in that setting is often shown for one generating set only, such as the set of Humphries generators. It follows that the topological notion of genericity is more canonical, as it is less dependent on such choices. In addition, it is suited to formulating problems about subgroups of mapping class groups. These properties can be understood in terms of subset topologies. An example of such a problem is the well known Ivanov conjecture, \cite{Ivanov}. When $\Sigma$ has genus at least 3, the Ivanov conjecture predicts that any finite index subgroup $\Gamma$ of $\mcg$ satisfies $H_{1}(\Gamma;\mathbb{R})=0$.\\

On the other hand, the measure theoretic version is better suited to asking questions about the asymptotic growth of invariants (stretch factor, volume of the mapping torus, etc...) of ``generic mapping classes''. It is possible to simply study the average value of the invariant of the $n$-step random walk $W_n$. \\

The paper is organised as follows: In Section \ref{sec:groups}, topologies $\mathcal{O}_{G,S}$ on groups $G$ equipped with a generating set $S$ are defined. The general properties of those topologies are studied and  it is shown that the twist-cofinite topology is obtained as an example of such a topology, taking $G=\mcg$ and $S$ to be the set of all Dehn-twists. Theorem \ref{thm:topology} and Proposition \ref{prop:topology} are corollaries of this more general theory. In Section \ref{sec:genericMCG}, generic properties of elements of $\mcg$ are discussed, and Theorem \ref{thm:pseudoAnosov} is proven. Sections \ref{sec:splitting} and \ref{sec:final} deal with properties of $3$-manifolds obtained from generic Heegaard splittings. The proof of Theorem \ref{thm:Hempel} requires concepts such as subsurface projections, due to Masur and Minsky, and properties of handlebody sets, due to Masur and Schleimer. This background material is given in Subsections \ref{sec:handlebody} and \ref{sec:subsurface}, as well as a crucial property of bounded and unbounded diameter subsurface projections of handlebody sets. The remainder of Section \ref{sec:final} deals mainly with showing the existence of certain large subsurface projections needed to obtain distance bounds in the proof of Theorem \ref{thm:Hempel}.

%%%%%%%%%%%%%%%
\subsection*{Acknowledgments} 
As noted above, the definition of the twist cofinite topology comes from Renaud Detcherry, and the first two subsections of this paper were also based on  \cite{Detcherry},  with the author's permission. The author would also like to thank Saul Schleimer for helpful discussions and an anonymous Referee for pointing out a number of useful references and devoting considerable time to writing a helpful commentary that led to improved clarity of exposition.

%%%%%%%%%
\section{A topology on groups with generating sets}
\label{sec:groups}
Let $G$ be a group and $S$ be a generating set for $G.$ For $g \in G,$ let $L_g$ be the left multiplication:
$$\begin{array}{rccl}L_g: & G & \longrightarrow & G
\\ & h & \longrightarrow & gh
\end{array}$$
and $R_g$ be the right multiplication:
$$\begin{array}{rccl}R_g: & G & \longrightarrow & G
\\ & h & \longrightarrow & hg.
\end{array}$$
Let also $s :G\rightarrow G$ be the map such that $s(g)=g^{-1}.$
Two topologies on $G$ are defined:
\begin{defn}\label{def:topo}Let
$$\mathcal{O}_{G,S}^l=\lbrace U \subset G \ | \ \forall g \in U ,\forall x \in S, \text{ For all except at most finitely many } n \in \Z \ , \ g x^n \in U\rbrace$$
and 
$$\mathcal{O}_{G,S}^r=\lbrace U \subset G \ | \ \forall g \in U, \forall x \in S, \text{For all except at most finitely many }  n \in \Z \  \ x^n g \in U\rbrace.$$
\end{defn}

Let $\mathcal{O}_{G,S}=\mathcal{O}_{G,S}^r \cap \mathcal{O}_{G,S}^l.$

\begin{thm} 
\label{generaltopologytheorem}
The sets $\mathcal{O}_{G,S}^l$ (resp. $\mathcal{O}_{G,S}^r$) determine a topology on $G$ for which the operator $L_g$ (resp. $R_g$)  is continuous for any $g \in G$. In addition, $G$ is a topological group for the topology $\mathcal{O}_{G,S}$, and $G$ is connected for all three topologies.
\end{thm}

The \textit{left (resp. right) cyclically cofinite topology} on $G$ relative to the generating set $S$ refers to the set of open sets $\mathcal{O}_{G,S}^l$ (resp. $\mathcal{O}_{G,S}^r$). The set of open sets $\mathcal{O}_{G,S}$ determines the \textit{cyclically cofinite topology} on $G$ relative to $S.$\\

The motivation for those names is that the intersection of an open set in $\mathcal{O}_{G,S}^l$ with a left coset of a cyclic subgroup of $G$ generated by an element of $S$ is either empty or cofinite.\\

In the proof of Theorem \ref{generaltopologytheorem}, the assumption that $S$ is a generating set is only used to show connectivity.

\begin{proof}
It will be shown that $\mathcal{O}_{G,S}^l$ is a topology; the case of $\mathcal{O}_{G,S}$ is analogous.\\

The empty set and $G$ are both in $\mathcal{O}_{G,S}^l.$ Suppose $U_1$, $U_2\in \mathcal{O}_{G,S}^l$, $g\in U_1\cap U_2,$ and let $x$ be in $S.$
It follows that $g x^n \in U_1$ except for at most finitely many $n \in \Z,$ and $g x^n \in U_2$ except for at most finitely many $n\in \Z.$ Consequently, $g x^n\in U_1 \cap U_2$ except for at most finitely many $n \in \Z.$ Therefore $U_1\cap U_2 $ is in $\mathcal{O}_{G,S}^l.$\\

Now it will be shown that an arbitrary union of open sets is open. Let $(U_i)_{i \in I}$ be a family of elements of $\mathcal{O}_{G,S}^l,$ and $U=\underset{i \in I}{\bigcup} U_i.$ If $g \in U$ and $x \in S,$ then $g\in U_i$ for some $i,$ and $g x^n \in U_i$ for all $n\in \Z$ except at most finitely many. In particular $g x^n \in U$ for all but finitely many $n.$ Therefore $\mathcal{O}_{G,S}^l$ is a topology.\\

It will now be shown that for any $g \in G,$ the operator $L_g$ is continuous (in the topology $\mathcal{O}_{G,S}^l$), or equivalently, for any open set $U,$ that $g^{-1}U$ is also open. Let $g^{-1}h \in g^{-1}U.$ As $U$ is open, for any $x \in S,$ all but finitely many $h x^n$ are in $U$ and all but finitely many $g^{-1}hx^n$ are in $g^{-1}U.$ Therefore $g^{-1}U$ is open.\\

Similarly, $R_g$ is continuous on $G$ for the topology $\mathcal{O}_{G,S}^r.$\\

To show connectivity, assume that $U$ is a nonempty open and closed subset of $G$ for the topology $\mathcal{O}_{G,S}^l.$ If $g\in U$ and $x\in S,$ then since $U$ is open, $g x^n \in U$ for all but at most finitely many $n$. In fact, $g x^n \in U$ for all $n\in \Z$. To see why, note that if $g x^i \in G\setminus U$ for some $i,$ since $G\setminus U$ is also open, $g x^n \in G\setminus U$ for all but finitely many $n\in \Z,$ hence one can find an integer $n$ such that $g x^n\in U \cap (G\setminus U)=\emptyset,$ giving a contradiction. As this is true for any $x$ in a generating set $S$, $U$ must be all of $G$. It follows that $G$ is connected for the topology $\mathcal{O}_{G,S}^l,$ and similarly also for the topology $\mathcal{O}_{G,S}^r.$\\

For the intersection topology $\mathcal{O}_{G,S},$ $G$ will also be connected and for any $g\in G,$ both $L_g$ and $R_g$ are continuous. To show that $G$ with the topology $\mathcal{O}_{G,S}$ is a topological group, it remains to show that the inversion map $s:G\rightarrow G$ is continuous, or equivalently, that for any open $U \in \mathcal{O}_{G,S},$  $U^{-1}\in \mathcal{O}_{G,S}.$\\

If $U \in \mathcal{O}_{G,S},$ $g^{-1} \in U^{-1}$ and $x\in S$, then $g x^n \in U$ and $x^n g \in U$ for all but finitely many $n\in \Z.$ It follows that $x^{-n}g^{-1}\in U^{-1}$ and $g^{-1}x^{-n} \in U^{-1}$ for all but finitely many $n\in \Z,$ hence $U^{-1} \in \mathcal{O}_{G,S}.$ Therefore $s$ is continuous and $G$ is a topological group, for the topology $\mathcal{O}_{G,S}.$
\end{proof}

Having constructed those topologies, it remains to be seen that they are interesting topologies. At the very least, one would want them to be non-trivial. It will be shown in Section \ref{sec:genericMCG} and \ref{sec:splitting} that in the example in which $G=\mcg$ for some compact connected oriented surface $\Sigma$ with $S$ being the set of all Dehn-twists, there are many interesting open sets. However, at this level of generality, the non-triviality of the topology $\mathcal{O}_{G,S}$ depends on the properties of the generating set $S,$ in particular the orders of elements in $S:$
\begin{prop}\label{prop:order} Let $G$ be a group and $S$ a generating set for $G.$
\begin{enumerate}
\item{If the generating set consists only of elements of infinite order, then $\mathcal{O}_{G,S}^l,\mathcal{O}_{G,S}^r$ and $\mathcal{O}_{G,S}$ are all finer than the cofinite topology on $G$.}
\item{In addition, if all pairs $x\neq y \in S$ satisfy $\langle x\rangle \cap \langle y \rangle =\lbrace 1 \rbrace,$ then $\mathcal{O}_{G,S}^l,\mathcal{O}_{G,S}^r$ and $\mathcal{O}_{G,S}$ are strictly finer.}
\item{If, on the other hand, the generating set consists only of elements of finite order, $\mathcal{O}_{G,S}^l,\mathcal{O}_{G,S}^r,\mathcal{O}_{G,S}$ are all the trivial topology on $G.$}
\end{enumerate}
\end{prop}

\begin{proof}
Both claims will be proven only for $\mathcal{O}_{G,S}^l$; the case of $\mathcal{O}_{G,S}^r$ being similar, and the case of $\mathcal{O}_{G,S}$ being a consequence of the first two cases by taking the intersection.\\

Assume first that $S$ contains only elements of infinite order. It is necessary to show that for any finite subset $V\subset G,$ the set $G\setminus V$ is open. Let $h \in G\setminus V$ and $x\in S.$ As $x$ has infinite order in $G,$ the elements $h x^n$ are all different, therefore all but at most finitely many are in $G\setminus V.$ Hence $G\setminus V$ is open, and $\mathcal{O}_{G,S}^l$ is finer than the cofinite topology on $G.$\\

Now assume also that $\langle x \rangle \cap \langle y \rangle =\lbrace 1 \rbrace$ for any two distinct elements of $S.$ It will now be shown that for any generator $x\in S,$ $G\setminus \langle x \rangle \in \mathcal{O}_{G,S}^l.$\\

If $g \in G\setminus \langle x \rangle$ then $gx^n \notin \langle x \rangle$ for any $n\in \Z.$ Suppose $y \in S$ is another generator. It follows that $gy^{n} \in \langle x \rangle$ for at most one integer $n\in \Z.$ If $gy^i, gy^j\in \langle x \rangle$ for $i\neq j,$ then $y^{j-i} \in \langle x \rangle \cap \langle y \rangle.$ But $y^{j-i}\neq 1$ as $y$ has infinite order which contradicts the fact that $\langle x \rangle \cap \langle y \rangle= \lbrace 1 \rbrace.$ \\

Finally, assume that $S$ consists only of finite order elements. Let $U \in \mathcal{O}_{G,S}^l$ be a non-empty open set, and let $g\in U$ and $x\in S.$ Assume that $x^d=1.$ It follows that all but at most finitely many elements $g x^{1+kd}$ where $k \in \Z$ are in $U.$ But $g x^{1+kd}=g x$ so $gx \in U.$ Similarly, one shows that $gx^{-1} \in U.$ \\

Therefore $U$ is closed under multiplication on the right by any element in $S$ or any inverse of an element in $S.$ As the set $S$ is a generating set of $G,$ this implies that that $U=G.$ Consequently $\mathcal{O}_{G,S}^l$ is the trivial topology in this case.
\end{proof}

The cyclically-cofinite topologies relative to a generating set are compatible with morphisms of groups, as the following proposition shows:

\begin{prop}\label{prop:morphism}Let $\phi :G \rightarrow H$ be a surjective morphism of groups and $S$ be a generating set of $G$. It follows that $\phi(S)$ is a generating set for $H,$ and for the topologies $\mathcal{O}_{G,S}^l$ and $\mathcal{O}_{H,\phi(S)}^l$ (or 
with $\mathcal{O}_{G,S}^r$ and $\mathcal{O}_{H,\phi(S)}^r$, or $\mathcal{O}_{G,S}$ and $\mathcal{O}_{H,\phi(S)}$) the map $\phi$ is continuous.
\end{prop}

\begin{proof}
The set $\phi(S)$ is a generating set for $H$ as $\phi$ is a surjective morphism. Once again the statement will be proven for the left cyclically-cofinite topologies only.\\

Let $U\in \mathcal{O}_{H,\phi(S)}^l,$ and let $g \in \phi^{-1}(U).$ For any $x\in S,$ it holds that $\phi(gx^n)=\phi(g)\phi(x)^n \in U$ for all but finitely many $n\in \Z,$ as $U$ is open. Therefore $\phi^{-1}(U)\in \mathcal{O}_{G,S}^l,$ and $\phi$ is continuous. 
\end{proof}

When the generating set $S$ is closed under conjugation by an arbitrary element of $G,$ the above topologies are better behaved:

\begin{prop}\label{prop:normal_gen_set} Assume that $S$ is a generating set for a group $G$ that is closed under conjugation by any element of $G.$ In this case the topologies $\mathcal{O}_{G,S}^r, \mathcal{O}_{G,S}^l$ and $\mathcal{O}_{G,S}$ all coincide.

Also, the intersection of any two non-empty open subsets of $G$ is a non-empty open subset of $G$, i.e. the topologies $\mathcal{O}_{G,S}^r, \mathcal{O}_{G,S}^l$ and $\mathcal{O}_{G,S}$ are not Hausdorff.
\end{prop}

\begin{proof}
Let $U\in \mathcal{O}_{G,S}^l$ and let $g \in U.$ It holds that for any $x\in S,$ the element $g^{-1}xg\in S$ as $S$ is closed under conjugation. Therefore $x^n g=g (g^{-1}x g)^n \in U$ for all but at most finitely many $n\in \Z.$ This shows that $U \in \mathcal{O}_{G,S}^r$ and $\mathcal{O}_{G,S}^l \subset \mathcal{O}_{G,S}^r.$\\

Similarly, $\mathcal{O}_{G,S}^r \subset \mathcal{O}_{G,S}^l$ and hence $\mathcal{O}_{G,S}^l=\mathcal{O}_{G,S}^r=\mathcal{O}_{G,S}.$\\

Now let $U_1$ and $U_2$ be two non-empty open sets of $G.$ For $g\in U_1$ and $h\in U_2,$ as $S$ is a generating set, one can write:
$$h=g x_1^{i_1} x_2^{i_2} \ldots x_k^{i_k}$$
for some $k\geqslant 0,$ some $x_j \in S$ and $i_j \in \Z.$ It will be shown that one can find $g\in U_1$ and $h\in U_2$ such that $k$ can be taken to be zero which implies that $g=h \in U_1\cap U_2.$\\

Let $g\in U_1$ and $h\in U_2.$ Since $U_2$ is open, $h x_k^n \in U_2$ for all but finitely many $n\in \Z.$ Also $g x_k^m \in U_1$ for all but finitely many $m \in \Z.$ In particular, one can choose $n$ and $m$ such that $h x_k^n \in U_2,$ $gx_k^m \in U_1$ and $m=i_k+n.$ In this case
$$h x_k^n=(g x_k^m) x_k^{-m} x_1^{i_1} \ldots x_k^{i_k+n}=(g x_k^m) x_k^{-m} x_1^{i_1} \ldots x_k^{m}=(g x_k^m) y_1^{i_1} \ldots y_{k-1}^{i_{k-1}}$$
where $y_j=x_k^{-m}x_j x_k^m \in S$ because $S$ is closed under conjugation. Therefore it is possible to inductively decrease $k$ by $1$ until $k=0$ is obtained, showing that $U_1\cap U_2\neq \emptyset.$
\end{proof}

Proposition \ref{prop:normal_gen_set} implies that if $S$ is a generating set with the property of being closed under conjugation by $G,$ then any non-empty open set in $\mathcal{O}_{G,S}$ is dense.\\

As in the introduction, call a subset $V \subset G$ \textit{thick} if it has non-empty interior, and \textit{thin} if its complement is thick. Any thick set is dense because its interior is dense, and any intersection of two thick sets is thick as the intersection of their interiors is a non-empty open set.\\

It will sometimes be necessary to consider sets that are dense but not thick. An example will be given in Theorem \ref{thm:Betti}.

\begin{lem} 
\label{lem:notthick}
Let $G$ be a group and $S$ be a generating set for $G$ that is closed under conjugation. Assume in addition that $V$ is a non-empty subset of $G$ such that for any $g \in V,$ and any $x\in S,$ infinitely many $g x^n$ are in $V.$
Under these assumptions $V$ is dense in $G$ for the topology $\mathcal{O}_{G,S}$.
\end{lem}

\begin{proof}
As in the proof of Proposition \ref{prop:normal_gen_set}, consider $g \in U$, where $U$ is a non-empty open set and $h \in V$ and $V$ satisfies the hypothesis of the proposition. In this case
$$h=g x_1^{i_1} x_2^{i_2} \ldots x_k^{i_k}$$ for some $k\geqslant 0$ and $x_j \in S.$ Since $h\in V$ the set $T=\lbrace m \in \Z \ | h x_k^m \in V \rbrace$ is infinite. Moreover $g x_k^n \in U$ for all but finitely many $n\in \Z,$ and hence there is some $n$ in $\lbrace m-i_k \ | m\in T \rbrace$ such that $g x_k^n\in  U.$ As in the previous proof it follows that one could find another pair $g'\in U, h'\in V$ with $h'=g' y_1^{j_1} \ldots y_{k-1}^{j_{k-1}}$ with all $y_l$'s in $S.$ By induction, it follows that $U \cap V \neq \emptyset.$  Hence $V$ has non-empty intersection with any open set, i.e. $V$ is dense. 
\end{proof}

The twist-cofinite topology defined in Theorem \ref{thm:topology} is an example of such a topology, where $G=\mcg$ for some compact connected oriented surface $\Sigma,$ and $S$ is the set of all Dehn-twists. The set of all Dehn-twists generates $\mcg$ and is closed under conjugation; if $f \in \mcg$ and $D_{\gamma}$ is a Dehn-twist around a curve $\gamma$ on $\Sigma$ then $f \circ D_{\gamma} \circ f^{-1}=D_{f(\gamma)}.$ Therefore Theorem \ref{thm:topology} and Proposition \ref{prop:topology} are consequences of the more general Definition \ref{def:topo} and Proposition \ref{prop:normal_gen_set} above.   

%%%%%%%%%%%%%
\section{Generic properties of pseudo-Anosovs}
\label{sec:genericMCG}
In this section, Theorem \ref{thm:pseudoAnosov} will be obtained as a corollary of classical results.

\begin{thm*}[Theorem \ref{thm:pseudoAnosov} of the introduction]
The following subsets of $\mcg$ are non-empty and open, and hence thick:
\begin{enumerate}
\item{The set of pseudo-Anosov maps $f \in \mcg.$}
\item{The sets $\lbrace f \in \mcg \ \textrm{pseudo-Anosov} \ | \ \Vol (M_{f})>C \rbrace,$ where $M_f$ is the mapping torus of $f$ and $C>0.$}
\item{The sets $\lbrace f \in \mcg \ \textrm{pseudo-Anosov} \ | \ \lambda(f)>C \rbrace,$ where $\lambda(f)$ is the stretch factor of $f$ and $C>0.$}
\item{The set of pseudo-Anosovs that do not leave any integer homology class of curves on $\Sigma$ invariant.}
\end{enumerate}
\end{thm*}

\begin{proof}
The first part of the theorem follows directly from the following
\begin{thm}[Theorem A  part $i$ of \cite{LM86}]
Let $f:\Sigma\rightarrow \Sigma$ be a pseudo-Anosov map, and $c$ a simple curve on $\Sigma$. It follows that for all but at most finitely many $n$, the composition $D_{c}^{n}f$ is also pseudo-Anosov. In Theorem 0.1 of \cite{fathi}, this theorem was strengthened to the statement that $D_{c}^{n}f$ is pseudo-Anosov for all but at most 7 consecutive values of $n$.
\end{thm}

A Dehn surgery on a 3-manifold $M$ is a construction whereby the interior $\mathring T_{1}$ of a solid torus $T_{1}$ in the interor $\mathring M$ of $M$ is removed, and another solid torus $T_{2}$ is glued in via a homeomorphism $h:\partial T_{2}\rightarrow \partial T_{1}$. The 3-manifold $(M\setminus \mathring T_{1})\cup_{h}T_{2}$ is said to be obtained from $M$ by a Dehn surgery.\\

To prove the second part of the Theorem, note that $M_{fD_{c}^{n}}$ is obtained from $M_{f}$ by a Dehn surgery. The behaviour of volume under Dehn surgeries has been studied in \cite{NZ}, where it follows from Theorem 1A that  for sufficiently large $|n|$
\begin{equation*}
\Vol (M_{f})<\Vol (M_{fD_{c}^{n}}).
\end{equation*}
Similarly for sufficiently large $|n|$, 
\begin{equation*}
\Vol (M_{f})<\Vol (M_{D_{c}^{n}f}).
\end{equation*}

The third part of the theorem also follows directly from a Theorem of \cite{LM86}. It is a consequence of Theorem C of \cite{LM86} that there are positive constants $K$, $L_{1}$ and $L_{2}$ such that 
\begin{equation*}
|n|L_{1}-K\leq \lambda(D_{c}^{n}f)\leq L_{2}|n|+K
\end{equation*}

Denote by $[c]$ the integer homology class in $H_{1}(\Sigma; \mathbb{Z})$ with representative $c$. If $[c]=0$, the action of the mapping class $f$ on homology is identical to that of $fD_{c}^{n}$ or $D_{c}^{n}f$. If $c$ is a nonseparating curve, choose a set of curves containing $c$ that determine a symplectic basis for $H_{1}(\Sigma; \mathbb{Z})$, and denote by $d$ the unique curve in the set intersecting $c$. Since the action of the mapping class group on $\pi_{1}(\Sigma)$ defines an action on $H_{1}(\Sigma; \mathbb{Z})$, the action of a mapping class on $H_{1}(\Sigma; \mathbb{Z})$ is determined by its action on the curves making up the chosen basis. It follows immediately that there is at most one value of $n$ for which $fD_{c}^{n}$ acts trivially on homology. Replacing the set of curves making up the basis by their images under $f^{-1}$, the same argument shows the same is true for $D_{c}^{n}f$. Therefore, the set of pseudo-Anosovs that do not leave any integer homology class of curves on $\Sigma$ invariant is open. This set is known to be non-empty, so the theorem follows.
\end{proof}

Since Dehn twists are not of finite order in the mapping class group, it follows immediately from Proposition \ref{prop:order} that the subgroups of the mapping class group in Theorem \ref{thm:pseudoAnosov} are also generic in finite index subgroups of mapping class groups. In \cite{Maher}, it was shown that there is a sense in which pseudo-Anosovs are generic in every finitely generated subgroup of $\mcg$ containing a pseudo-Anosov element. Apart from this, little seems to be known about subset topologies.

%%%%%%%%%%%

%\section{The topology of subgroups of $\mcg$}
%\label{sec:subgroups}

%%%%%%%
\section{Heegaard splittings of $3$-manifolds with generic gluing maps}
\label{sec:splitting}
In this section, the homology of 3-manifolds described as Heegaard splittings with generic gluing maps will be studied. A reference for any facts about Heegaard splittings used here is \cite{notes}. \\

Suppose a 3-manifold $M$ is decomposed as a Heegaard splitting $N_{f}=H_{A}\cup H_{B}$ along a surface $\Sigma$. The homology of $M$ can be computed from the Heegaard splitting, using Mayer-Vietoris. If $g$ is the genus of $\Sigma$, then
\begin{equation*}
\label{MV}
H_{2}(N_{f}; \mathbb{Z})\rightarrow H_{1}(\Sigma; \mathbb{Z})=\mathbb{Z}^{2g}\rightarrow H_{1}(H_{A}; \mathbb{Z})\bigoplus H_{1}(H_{B};\mathbb{Z}) =\mathbb{Z}^{g}\bigoplus \mathbb{Z}^{g}\rightarrow H_{1}(N_{f};\mathbb{Z})
\end{equation*}
To compute $H_{2}(N_{f}; \mathbb{Z})$ and $H_{1}(N_{f}; \mathbb{Z})$ it is therefore necessary to calculate the $2g\times 2g$ matrix $K(f)$ describing the map $H_{1}(\Sigma; \mathbb{Z})\rightarrow H_{1}(H_{A}; \mathbb{Z})\bigoplus H_{1}(H_{B};\mathbb{Z})$.\\

Let $\{x_{1},\ldots,x_{2g}\}$ be a generating set for $\pi_{1}(\Sigma)$. For $H_{A}$ let $a_{1}, \ldots, a_{g}$ be a set of curves on $\Sigma$ that bound a set of disks in $H_{A}$, such that when $H_{A}$ is cut along these disks, a ball is obtained. The curves $\{b_{1}, \ldots, b_{g}\}$ are defined similarly for $H_{B}$. Each curve $a_{i}$ or $b_{i}$ for $i=1,\ldots, g$ is described by a word in the generators $\{x_{1},\ldots,x_{2g}\}$. The first row of $K(f)$ is the vector representing $[a_{1}]$ in the basis $\{[x_{1}],\ldots,[x_{2g}]\}$, the second row is the vector representing $[a_{2}]$ followed by vectors representing $[a_{3}],\ldots, [a_{g}], [b_{1}], \ldots, [b_{g}]$ in the given order.

\begin{lem}[Lemma 3.31 of \cite{notes}]
\label{lem:rank}
$H_{1}(N_{f};\mathbb{Z})$ is finite iff the determinant of $K(f)$ is nonzero. In addition, if $H_{1}(N_{f};\mathbb{Z})$ is finite then the order of  $H_{1}(N_{f};\mathbb{Z})$ is equal to the absolute value of the determinant of $K(f)$.
\end{lem}

\begin{thm}[Theorem \ref{thm:Betti} of the introduction]
Suppose $b_{1}(N_{f},\mathbb{Z})$ is the first Betti number of $N_{f}$,  and  $|H_{1}(N_{f};\Z)|$ is the number of elements in the group $H_{1}(N_{f};\Z)$.  In this case 
\begin{enumerate}
\item The set $\lbrace f \in \mcg \ | \ b_{1}(N_{f},\mathbb{Z})=0\rbrace$ is a thick set
\item The set $\lbrace f \in \mcg \ | \ |H_{1}(N_{f};\Z)|\geqslant k \rbrace$ for any integer $k\geqslant 0$ is a thick set
\item If $p$ is prime and $0\leqslant i \leqslant g$, the set $\lbrace f \in \mcg \ | \ b_{1}(N_{f},\mathbb{F}_p)=i\rbrace$ is dense but not thick.
\end{enumerate}
\end{thm}

\begin{proof}
By Lemma \ref{lem:rank}, proving the first part of the theorem is the same as showing that the set $\lbrace f \in \mcg \ | \ \det(A(f))\neq0\rbrace$ is thick. Since this set is clearly nonempty, it suffices to show that the set is open.
Consider the map from $K(f)$ to $K(fD_{c}^{n})$. This map fixes the last $g$ rows, and adds the vector $n\hat{i}(a_{i}, c)[c]$ to row $i$, for $1\leq i\leq g$. Here $\hat{i}$ denotes algebraic intersection number, and $[c]$ is assumed to be a vector representing the homology class with representative $c$ in the basis $\{[x_{1}],\ldots,[x_{2g}]\}$.\\

Since the determinant of $K(f)$ is nonzero, $[a_{1}], \ldots, [a_{g}]$ and $[b_{1}], \ldots, [b_{g}]$ can be chosen such that $[c]$ is in the span of $[a_{g}]$ and $[b_{1}]$.  It follows that the determinant of $K(fD_{c}^{n})$ can only be zero if  the projection of $n\hat{i}(a_{i}, c)[c]$ onto the unit vector parallel to $[a_{g}]$ is given by $-[a_{g}]$. This can happen for at most one value of $n$.\\

Similarly, consider the  map from $K(f)$ to $K(D_{c}^{n}f)$. This map fixes the first $g$ rows, and adds the vector $n\hat{i}(b_{i}, c)[c]$ to row $i+g$, for $1\leq i\leq g$. The same argument as before shows that the determinant of $K(D_{c}^{n}f)$ can only be zero for finitely many $n$. This proves the first part of the theorem.\\

From the expressions for $K(fD_{c}^{n})$ and $K(D_{c}^{n}f)$ just given, it follows that for $n^{*}$ sufficiently large, $n>n^{*}$ implies that $|\det(K(fD_{c}^{n}))|\geq |\det(K)|$ and $|\det(K(D_{c}^{n}f))|\geq |\det(K)|$. The second part of the theorem is then a consequence of Lemma \ref{lem:rank}.\\

To prove the last part of the theorem, note that $b_{1}(N_{f}, \mathbb{F}_p)$ is equal to the dimension of the kernel of $K(f) \mod p$. Similarly for $b_{1}(N_{fD_{c}^{n}}, \mathbb{F}_p)$ and  $b_{1}(N_{D_{c}^{n}f}, \mathbb{F}_p)$. Since $K(fD_{c}^{np})=K(D_{c}^{np}f)=K(f)\mod p$, the theorem follows from Lemma \ref{lem:notthick}.
\end{proof}

\section{Hempel distance  and thick sets.}
\label{sec:final}
This final section uses the notion of Hempel distance to prove further theorems about 3-manifolds obtained as Heegaard splittings with generic gluing maps. Subsection \ref{sec:Hempel} defines Hempel distance, and shows how it relates to the study of properties of 3-manifolds describable as Heegaard splittings with generic gluing maps. Subsection \ref{sec:handlebody} covers some background on handlebody sets. In Subsection \ref{sec:subsurface}, subsurface projections are introduced, and subsurfaces to which handlebody sets have unbounded diameters in subsurface projections are studied. The final subsection uses these results to partition the handlebody sets into 3 subsets. This partition is used to give a proof of Lemma \ref{maintheorem}, and hence of Theorems  \ref{thm:stabilisation}, \ref{thm:incompressible} and \ref{thm:Hempel}. 

\subsection{Incompressible surfaces and Hempel distance}
\label{sec:Hempel}
Hempel distance is defined in terms of Harvey's curve complex $\mathcal{C}(\Sigma)$. This is the simplicial complex with vertices corresponding to nontrivial homotopy classes of simple closed curves on $\Sigma$, and simplices corresponding to sets of curves with pairwise geometric intersection number zero. There is a combinatorial distance defined on $\mathcal{C}(\Sigma)$, where the distance between two vertices is equal to the smallest number of edges of any path connecting the vertices. A Heegaard splitting $N_{f}=H_{A}\cup H_{B}$ determines two sets in $\mathcal{C}(\Sigma)$. One of the sets consists of the set of simple curves on $\Sigma=\partial H_{A}=\partial H_{B}$ that bound disks in $H_{A}$, and the other set consists of the set of simple curves on $\Sigma$ that bound disks in $H_{B}$. The former set will be called $A$ and the latter $B$. The Hempel distance, \cite{Hempel}, is defined to be the distance $d(A,B)$ in $\mathcal{C}(\Sigma)$ between the two sets $A$ and $B$, where $d(A,B):=\inf_{a\in A, b\in B}d(a,b)$.\\

 It is argued that the Hempel distance is an interesting measure of the complexity of $N_{f}$, for example, because reducibility properties of $N_{f}$ can be elegantly formulated in terms of Hempel distance \cite{CassonGordon}. In the context of this paper, Hempel distance will be used to bound from below the genus of an incompressible surface embedded in $N_{f}$. \\

\begin{lem}[Theorem 1.2 of \cite{Har02}, or consequence of Theorem 3.1 \cite{BS}]
\label{lem:genusdistance}
The genus of a closed, incompressible surface $F$ embedded in $N_{f}$ is greater than or equal to half the Hempel distance.
\end{lem}

An important special case of Lemma \ref{lem:genusdistance} in \cite{Hempel} is that if the Hempel distance of $N_{f}$ is greater than or equal to 3,  $N_{f}$ cannot contain an incompressible sphere or torus, from which it is then shown that $N_{f}$ is hyperbolic.\\

This final section will be primarily devoted to proving the following lemma:
\begin{lem}
\label{maintheorem}
Suppose that $d(A,B)\geq 3$, and let $c$ be a simple curve. It follows that for all but finitely many $n\in \mathbb{Z}$, 
\begin{equation*}
d(D_{c}^{n}(A), B)\geq d(A,B).
\end{equation*}
\end{lem}

This lemma implies Theorem \ref{thm:Hempel}. Together with Lemma \ref{lem:genusdistance}, this proves Theorem \ref{thm:incompressible}.\\

It will now be shown that Theorem \ref{thm:stabilisation} also follows from Lemma \ref{maintheorem}.

\begin{thm}
Assuming Lemma \ref{maintheorem}, the  sets $\lbrace f \in \mcg \ | \ g(N_f)=g(\Sigma) \rbrace$ and $\lbrace f \in \mcg \ | \ g(N_{f})\geq n\rbrace$ for $ n\geq 2$ are thick.
\end{thm}

\begin{proof}
When $g(N_{f})<g(\Sigma)$ the Heegaard splitting is said to be stabilised. By Lemma 5.5 of \cite{notes}, this happens iff there is an embedded disk $d_{a}$ in $H_{A}$ and an embedded disk $d_{b}$ in $H_{B}$ such that the boundaries of the disks are curves on $\Sigma$ in general position that intersect once only.\\

A necessary condition for the Heegaard splitting to be stabilised is therefore that $d(A,B)\leq 2$. By Lemma \ref{maintheorem}, the set $$\lbrace f \in \mcg \ | \text{ the Heegaard splitting induced by } f \text{ has Hempel distance}\geq 3 \rbrace$$ is thick. The intersection of the set $$\lbrace f \in \mcg \ | \ g(N_f)=g(\Sigma) \rbrace$$ with the thick set $$\lbrace f \in \mcg \ | \text{ the Heegaard splitting induced by } f \text{ has Hempel distance}\geq 3 \rbrace$$ is the entire set $$\lbrace f \in \mcg \ | \text{ the Heegaard splitting induced by } f \text{ has Hempel distance}\geq 3 \rbrace$$ and hence thick. That $\lbrace f \in \mcg \ | \ g(N_f)=g(\Sigma) \rbrace$ is thick then follows from Proposition \ref{prop:topology}.\\

For compact, orientable 3-manifolds with Heegaard splitting surfaces $P$ and $Q$, in \cite{ST06} it was shown that either $Q$ is isotopic to a possibly stabilised copy of $P$, or the Hempel distance of the splitting with splitting surface $P$ is bounded from above by twice the genus of $Q$. It follows from the argument just given, Proposition \ref{prop:topology}, Lemma \ref{maintheorem} and the first statement of this theorem that the set  $\lbrace f \in \mcg \ | \ g(N_{f})\geq n, \text{ for } n\geq 2\rbrace$ is thick.
\end{proof}

\subsection{Handlebody sets}
\label{sec:handlebody}

Recall that  $H_{A}$ and $H_{B}$ are handlebodies, with $\partial H_{A}$ and $\partial H_{B}$ both given by the surface $\Sigma$. The set of curves on $\partial H_{A}$ that bound embedded disks in $H_{A}$ will be denoted by $A$, and the set of curves on $\partial H_{B}$ that bound embedded disks in $H_{B}$ by $B$. Sets of vertices in $\mathcal{C}(\Sigma)$ that arise in this way are called handlebody sets. This subsection gives some background on handlebody sets that will be needed in the proof of Lemma \ref{maintheorem}.\\
 
\textbf{Basic assumptions and notations.} A curve is a nontrivial isotopy class of unpointed embeddings of $S^1$ into $\Sigma$.  Arcs in a subsurface of $\Sigma$ are isotopy classes, where isotopies are required to keep the endpoints on the boundary of the subsurface. Where convenient, a curve or arc will sometimes be confused with the image of a particular representative of the isotopy class, or the corresponding vertex of $\mathcal{C}(\Sigma)$. For example, when cutting a (sub)surface along a curve or arc, it will always be assumed that ``curve'' or ``arc'' refers to the image in the surface of a particular representative of the isotopy class, such as a geodesic.  A curve or arc will be said to be contained in a subsurface if there is a representative of the homotopy class with image contained in the subsurface.\\

The geometric intersection number of two curves, $c_{1}$ and $c_{2}$, on $\Sigma$ will be denoted by $i(c_{1}, c_{2})$. The geometric intersection number of two arcs $c_{1}$ and $c_{2}$ will also be denoted by $i(c_{1}, c_{2})$. Recall that $\hat{i}(c_{1},c_{2})$ denotes algebraic intersection number.\\

The most difficult case in Lemma \ref{maintheorem} occurs when the curve $c$ is distance 1 from $A$ or $B$. The reason for this will become apparent later. From now on, unless explicitly stated otherwise, it will be assumed that the curve $c$ in the statement of Lemma \ref{maintheorem} is distance 1 from $A$. The general case, in which $c$ could be any distance from $A$, will be dealt with in Lemma \ref{maintheorem}.\\
 
The set $A$ contains an unoriented set of $3g-3$ curves $\{a_{i}\}$ that determine a pants decomposition of $\Sigma$. It will be assumed that the curves $\{a_{1}, \ldots, a_{m-1}\}$ are disjoint from $c$, and the curves $\{a_{m}, \ldots, a_{n}\}$ are not. Since $c$ is distance 1 from $A$, $m-1\neq n$ and the curves can be chosen such that $m>1$. \\

The embedded disk bounded by a curve $a\in A$ will be denoted by $\mathcal{D}_{a}$.\\

Curves in a handlebody set can be conveniently described in terms of a certain decomposion, which will now be defined.\\

\textbf{Band sums.} Suppose $d$ and $e$ are two simple, disjoint curves on $\Sigma$, and $h$ is an embedded arc in $\Sigma \setminus (d\cup e)$ with one endpoint on $d$ and the other on $e$. The band sum of $d$ and $e$ is the connected component of the boundary of a regular neighbourhood of $d\cup h \cup e$ that is not isotopic to either $d$ or $e$. This is illustrated in Figure \ref{bandsum}. \\

\begin{figure}[!thpb]
\centering
\includegraphics[width=0.6\textwidth]{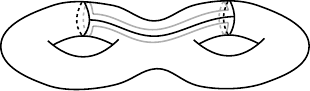}
\caption{A band sum of two curves on $\Sigma$, shown in grey.}
\label{bandsum}
\end{figure}

The band sum of two curves in $A$ is also in $A$. The band sum of $d$ and $e$ bounds a disk obtained by taking the (necessary disjoint) disks $\mathcal{D}_{d}$ and $\mathcal{D}_{e}$ and attaching a long, narrow rectangle $r$, with one ``short side'' on $\partial \mathcal{D}_{d}$ and the other ``short side'' on $\mathcal{D}_{e}$. The rectangles are attached in such a way that the orientations match up to give an oriented disk in the closure of the handlebody. Push the interior of the oriented disk into the interior of $H_{A}$, to obtain a disk with boundary the band sum of $d$ and $e$.\\

\begin{figure}[!thpb]
\centering
\includegraphics[width=0.2\textwidth]{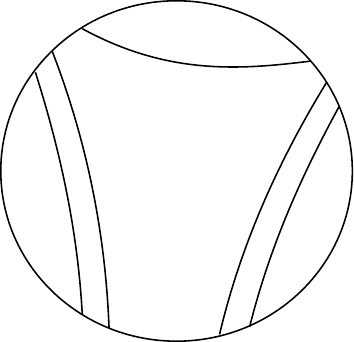}
\caption{The intersections of the disk $\mathcal{D}_{d_{i}}$ with $\mathcal{D}_{d_{j}}$.}
\label{diskintersections}
\end{figure}

\textbf{Bigon surgery.}  A pair of disks $\mathcal{D}_{d_{1}}$ and $\mathcal{D}_{d_{2}}$ embedded in $H_{A}$ with boundary curves $d_{1}$ and $d_{2}$ on $\partial H_{A}$ can be put in general and minimal position. The intersections consist of a finite number of embedded arcs with endpoints on the crossings of $d_{1}$ and $d_{2}$. These arcs cut each of the disks into pieces, as shown in Figure \ref{diskintersections}. Due to the fact that the sets of arcs are embedded in each of the disks, cutting each disk along the intersections must give at least two bigons for each disk. \\

Let $x$ be a connected component of the intersection of $\mathcal{D}_{d_{1}}$ with $\mathcal{D}_{d_{2}}$, with the property that $x$ makes up one side of a bigon in $\mathcal{D}_{d_{2}}$. Cut $\mathcal{D}_{d_{1}}$ along $x$, and glue in two copies of the bigon from $\mathcal{D}_{d_{2}}$; each with a different orientation. This gives two new embedded disks in $H_{A}$ with boundaries in $\partial H_{A}$. These disks are each disjoint from $\mathcal{D}_{d_{1}}$ and together have fewer crossings with $\mathcal{D}_{d_{2}}$ than $\mathcal{D}_{d_{1}}$. These two disks are obtained from the disk $\mathcal{D}_{d_{1}}$ by a bigon surgery.  

\begin{rem}
\label{algebraicint}
Any pair of curves $d$ and $e$ in $A$ have algebraic intersection number zero. This is because they bound disks that intersect along arcs in the interior of the handlebody. Each of these arcs has a pair of endpoints that are points of intersection of $d$ and $e$ with opposite handedness.
\end{rem}

For sufficiently small $\epsilon$, each boundary component of the $\epsilon$-neighbourhood of the union of geodesics $d_{1}\cup d_{2}$ is either homotopically trivial or a curve in $A$. Assuming that the disks $\mathcal{D}_{d_{1}}$ and $\mathcal{D}_{d_{2}}$ are in minimal position, the disk bounded by a boundary component is a boundary component of the $\epsilon$-neighbourhood of $\mathcal{D}_{d_{1}}\cup \mathcal{D}_{d_{2}}$ in $H_{A}$.

\subsection{Subsurface projections of handlebody sets}
\label{sec:subsurface}

This subsection studies subsurfaces to which projections of handlebody sets have unbounded diameter.\\

An embedded subsurface of $\Sigma$ will be called essential if all its boundary curves are homotopically nontrivial in $\Sigma$. Subsurfaces will be assumed to be compact, connected, essential, embedded subsurfaces of $\Sigma$. \\

In order to facilitate inductive arguments about distances and coarse geometry of $\mathcal{C}(\Sigma)$, Masur and Minsky \cite{MasurMinsky} defined the distance between two curves in a subsurface projection. Distances in subsurface projections to annuli are a special case of this definition. Distances between two curves $c_{1}$ and $c_{2}$ in the subsurface projection to a subsurface $Y$ will be denoted by $d_{Y}(c_{1}, c_{2})$. When discussing subsurface projections, in addition to the assumptions listed above, it will also be assumed that the subsurface $Y$ to which the projection is made is not a 3-holed sphere. This is because it was shown that large subsurface projections to 3-holed spheres never arise; their arc and curve complexes have bounded diameter. This paper will not give a detailed exposition of subsurface projections, as the topological nature of the arguments only require qualitative statements, rather than the strong numerical and/or uniform bounds for which the full force of \cite{MasurMinsky} is required. Interested readers not already familiar with the concept of subsurface projections are referred to \cite{MasurMinsky}.\\

Let $\mathcal{A}_c$ be an annulus with core curve $c$ embedded in $\Sigma$. Distances in the subsurface projection to $\mathcal{A}_{c}$ are a measure of the number of times one curve has been twisted around $c$ relative to the other. For example, as stated in Equation 2.6 of \cite{MasurMinsky}
\begin{equation*}
d_{\mathcal{A}_c}(a, D_{c}^{n}(a))=
\begin{cases}
2+|n|, & \text{ if } d(a,c)>1 \text{ and } n\neq 0 \\
0, & \text{ otherwise}
\end{cases}
\end{equation*}

%Instead of performing Dehn twists on curves, it will sometimes be necessary to modify only some subarcs of a curve. To twist a curve around $c$ means to choose a single arc of the curve passing through the annulus with core curve $c$, and to Dehn twist that arc around $c$. Figure \ref{increasetwisting} shows a curve that has been twisted around a curve $c$ and around a curve $t_{1}$. A twist around a curve $c$ will be denoted $T_{c}$. When there are choices involved about which arc to twist around, these choices will either be unimportant or will be clear from the context.\\

%\textbf{Twists and self-intersections.} When a curve has intersection number with $c$ greater than 1, applying $T_{c}$ will give rise to self crossings. In this paper, twists will only be performed in pairs that result in simple closed curves.\\

Distances in subsurface projections are used to obtain information about geodesic paths in the curve complex. The simplest way of doing this is based on the idea that if $d_{Y}(a,b)$ is sufficiently large relative to $d(a,b)$, any geodesic in $\mathcal{C}(\Sigma)$ from $a$ to $b$ must pass through a vertex corresponding to a curve disjoint from $Y$. An even stronger form of this statement is given in the following theorem.

\begin{thm}[Theorem 3.1 of \cite{MasurMinsky}]
\label{MM}
Suppose $Y$ is a subsurface of $\Sigma$. Let $\gamma$ be a geodesic segment, ray or infinite line in $\mathcal{C}(\Sigma)$, such that no vertex of $\gamma$ is disjoint from $Y$. There is a constant $\beta(\Sigma)$ depending only on $\Sigma$ such that any two vertices $v_{1}$ and $v_{2}$ of $\gamma$ satisfy $d_{Y}(v_{1}, v_{2})\leq \beta(\Sigma)$.
\end{thm}

The next proposition is a consequence of the previous theorem, and is included here to illustrate a ``hinge-like'' construction that will be used repeatedly in the proof of Lemma  \ref{maintheorem}. The proposition is also a consequence of Theorem 1.4 of \cite{Yoshizawa}.

\begin{prop}
\label{twocurves}
Let $a$, $b$, and $c$ be curves in $\Sigma$. It follows that for all but finitely many $n\in \mathbb{Z}$, 
\begin{equation*}
d(D_{c}^{n}(a), b)\geq d(a,b)
\end{equation*}
\end{prop}

\begin{proof}
Two basic but crucial observations are the following: The mapping class group (in this case, specifically powers of Dehn twists around $c$) acts by isometry on $\mathcal{C}(\Sigma)$. Secondly, $D_{c}$ fixes a vertex iff the vertex is distance at most one from $c$.\\

%Another simple but important observation is that if two simple curves, $c_{1}$ and $c_{2}$, both have essential intersections with $c$, a necessary condition for $c_{1}$ and $c_{2}$ to be disjoint is that $d_{\mathcal{A}_c}(c_{1}, c_{2})\leq 1$.\\

If $d(a,c)\leq 1$, then $D_{c}^{n}(a)=a$, so $d(D_{c}^{n}(a),b)=d(a,b)$. If $d(b,c)\leq 1$, then $D_{c}(b)=b$, so $d(a,b)=d(D^{n}_{c}(a),D^{n}_{c}(b))=d(D^{n}_{c}(a), b)$ and the lemma is again trivially true. It is therefore only necessary to deal with the case in which $c$ has essential intersections with both $a$ and $b$.\\

The construction used in the rest of the proof is illustrated in Figure \ref{hinge}. By Theorem \ref{MM}, there is an $n^{*}\in \mathbb{N}$ such that for $|n|\geq n^{*}$, every geodesic path in $\mathcal{C}(\Sigma)$ from $D_{c}^{n}(a)$ to $b$ passes through a vertex (this vertex depends on the path) corresponding to a curve disjoint from $c$. When $n=n^{*}$, denote this vertex by $c'$. In Figure \ref{hinge}, $\gamma_{n^{*}}$ is a geodesic from $D_{c}^{n^{*}}(a)$ to $b$. As shown in Figure \ref{hinge}, $\gamma_{n^{*},1}$ is used to denote the restriction of $\gamma_{n^{*}}$ from $D_{c}^{n^{*}}(a)$ to $c'$, and $\gamma_{n^{*},2}$ the restriction of $\gamma_{n^{*}}$ from $c'$ to $b$. Since $D_{c}^{-n^{*}}$ acts on $\mathcal{C}(\Sigma)$ by isometry, fixing $c'$, a path from $a$ to $b$ with the same length as $\gamma_{n^{*}}$ is obtained by concatenating $D_{c}^{-n^{*}}(\gamma_{n^{*},1})$ with $\gamma_{n^{*},2}$. This path cannot be shorter than a geodesic from $a$ to $b$, showing that $d(D_{c}^{n^{*}}(a), b)\geq d(a,b)$.\\

\begin{figure}[!thpb]
\centering
\includegraphics[width=0.45\textwidth]{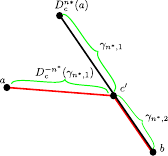}
\caption{Path construction in the proof of Proposition \ref{twocurves}. The path $\gamma_{n^{*}}$ is a geodesic from $D_{c}^{n^{*}}(a)$ to $b$, and $c'$ is a curve disjoint from $c$. The geodesic $\gamma_{n^{*}}$ is used to construct a path from $a$ to  $b$ as shown in red.}
\label{hinge}
\end{figure}

An analogous construction works for any $n\in \mathbb{Z}$ with $|n|>n^{*}$, proving the proposition.
\end{proof}

\begin{rem}
The construction in the proof of Proposition \ref{twocurves} implies that for all sufficiently large $|n|$, there exists a geodesic from $D_{c}^{n}(a)$ to $b$ passing through $c'$. This is proven simply by contradiction, as follows: Choose $n$ such that $|n|$ is large enough to ensure that every geodesic from $D_{c}^{n}(a)$ to $b$ passes through a vertex representing a curve disjoint from $c$, and suppose that no geodesic from $D_{c}^{n}(a)$ to $b$ passes through the vertex $c'$. This means that the path from $D_{c}^{n}(a)$ to $b$ obtained by concatenating $D_{c}^{(n-n^{*})}(\gamma_{n^{*},1})$ with $\gamma_{n^{*},2}$ is longer than a geodesic. However, as this path has the same length as $\gamma_{n^{*}}$, a geodesic $\gamma_{n}$ from $D_{c}^{n}(a)$ to $b$ can then be used to construct a path from $D_{c}^{n^{*}}(a)$ to $b$ shorter than the geodesic $\gamma_{n^{*}}$, which is the required contradiction.
\end{rem}

%Let $c'$ be a vertex satisfying $d(c,c')\leq 1$. For any $n$, a path $\gamma_{n}$ from $D_{c}^{n}(a)$ to $b$ of length at most $d(a,c)+d(c,b)$ can always be constructed, as follows: Start with a path $\gamma_{0}$ from $a$ to $b$ obtained by joining two paths; one from $a$ to $c'$, and the other from $c'$ to $b$. These paths can be chosen so that $\gamma_{0}$ has length less than or equal to $d(a,c)+d(c,b)$. Since $D_{c}^{n}$ fixes $c'$ and acts by isometry, a path $\gamma_{n}$ of the same length from $D_{c}^{n}(a)$ to $b$ is obtained by replacing the subpath from $a$ to $c'$ by its image under $D_{c}^{n}$.\\

%Now let $\gamma_{n^{*}}$ be any geodesic path in $\mathcal{C}(\Sigma)$ from $D_{c}^{n^{*}}(a)$ to $b$. When $c$ intersects each curve representing a vertex of $\gamma_{n^{*}}$, each edge of the path $\gamma_{n^{*}}$ can reduce the number of twists around $c$ by at most one. It follows that when $n^{*}$ is larger than $d(a,c)+d(c,b)$, there must be a vertex $c'$ on $\gamma_{n^{*}}$ within distance at most 1 from $c$. Moreover, since the image under $D^{n^{*}}_c$ of the curves representing the subpath of $\gamma_{n^{*}}$ connecting $D_{c}^{n*}(a)$ to $c'$ coincide outside of $\mathcal{A}_{c}$, for any $i$ such that $d_{\mathcal{A}_{c}}(D_{c}^{i}(a), b)\geq d(a,b)$ it is possible to choose $c'$ and the set of geodesic paths $\{\gamma_{i}\}$ such that every $\gamma_{i}$ passes through the same vertex $c'$.\\

%Since the distance $d(a,b)$ is no larger than the length of the path $\gamma_{0}$, this concludes the proof of the lemma.
%\end{proof}

\textbf{Large subsurface projections in handlebody sets.} Subsurfaces in which handlebody sets have unbounded diameter in subsurface projections will now be characterised. A tool that will be used to do this is an $I$-bundle. An $I$-bundle is a fiber bundle with fiber given by an interval $I$. Suppose the $I$-bundle is contained in $H_{A}$ and the base space $Y$ is an orientable surface with boundary. The boundary of the $I$-bundle consists of a horizontal boundary and a vertical boundary. The horizontal boundary is the union of the boundaries of the fibers. It consists of two disjoint, embedded subsurfaces $Y_{1}$ and $Y_{2}$ of $\Sigma$ both of which are homeomorphic to $Y$. The vertical boundary consists of boundary points in the interiors of fibers, and is a union of open annuli; one for each closed curve on $\partial Y$. The $I$-bundle will be called essential if $Y_{1}$ and $Y_{2}$ are essential subsurfaces of $\Sigma$ and each connected component of the vertical boundary is either contained in $\Sigma$, or the two curves on its boundary are not homotopic in $\Sigma$.

\begin{thm}[Corollary of Theorem 1.1 of \cite{MasurSchleimer}, as stated in \cite{Minsky}]
\label{masurschleimer}
The diameter of the subsurface projection of $A$ to a subsurface $Y$ of $\Sigma$ is bounded from above by a number $\delta(g)$ depending on the genus $g$ of $\Sigma$, unless 
\begin{enumerate}
\item{there is an element of $A$ in the complement of $Y$, }
\item{there is an element of $A$ in $Y$,} 
\item{there is an essential $I$-bundle $I$ in $H_{A}$ with $Y$ a component of its horizontal boundary, and at least one vertical annulus of $I$ lying in $\Sigma$.} 
\end{enumerate}
\end{thm}

The reason for the assumption that $d(A,B)\geq 3$ in Lemma \ref{maintheorem} comes from the next lemma.\\

\begin{lem}
\label{distance3}
Let $A$, $B$ and $\Sigma$ be as defined above, where $d(A,B)\geq 3$. Let $Y$ be a subsurface of $\Sigma$ to which $A$ has unbounded diameter in the subsurface projection. In this case $B$ has diameter bounded by $\delta(g)$ in the subsurface projection to $Y$. The same is true with $A$ and $B$ interchanged.
\end{lem}
\begin{proof}
The lemma will first be proven for subsurface projections to annuli, and then for other subsurfaces.\\

When $Y$ is an annulus, condition 3 of Theorem \ref{masurschleimer} does not apply, because the assumption that at least one vertical annulus of $I$ is in $\Sigma$ would prevent $I$ from being essential. In this case, the lemma can be seen to follow directly, because if there is both a curve $a\in A$ contained in $Y$ or $\Sigma\setminus Y$ and a curve $b\in B$ contained in $Y$ or $\Sigma\setminus Y$, this would imply that $d(a,b)\leq 2$, contradicting the assumption that $d(A, B)\geq 3$. Similarly, if $Y$ is not an annulus, whenever the third item of Theorem \ref{masurschleimer} is ruled out for both $H_{A}$ and $H_{B}$, the lemma holds for the subsurface $Y$.\\

Now suppose $Y$ is not an annulus, and $Y$ is a component of the horizontal boundary of an essential $I$-bundle in $H_{A}$, as in condition 3 of Theorem \ref{masurschleimer}. It will first be shown that $Y$ cannot also be a component of the horizontal boundary of an essential $I$-bundle in $H_{B}$, as this would imply $d(A,B)\leq 2$. Denote by $y_{1}$ a boundary component of $Y$ on the boundary of a vertical annulus of $I$ lying in $\Sigma$. Since $Y$ is not an annulus, there is an essential arc $\gamma_{1}$ contained in $Y$ with endpoints on $y_{1}$. A disk $\mathcal{D}_{a}$ embedded in $H_{A}$ is obtained by taking the union of the fibers of $I$ that each have an endpoint on $\gamma_{1}$. It follows that $\partial \mathcal{D}_{a}$ is a curve $a$ in $A$. The curve $a$ is disjoint in $\Sigma$ from a curve $\gamma$ in $Y$ obtained as a concatenation of $\gamma_{1}$ with a subarc of $y_{1}$ joining the endpoints on $y_{1}$ of $\gamma_{1}$. If $Y$ is also a component of the horizontal boundary of an essential $I$-bundle in $H_{B}$, a curve $b$ in $B$ disjoint from $\gamma$ will now be constructed. If $y_{1}$ is also a boundary component of a vertical annulus in $\Sigma$ of an essential $I$-bundle in $H_{B}$, $b$ is the boundary of the disk obtained by taking the union of the fibers of the $I$-bundle in $H_{B}$ with an endpoint on $\gamma_{1}$. Otherwise, denote by $y_{2}$ a boundary component of a vertical annulus in $\Sigma$ of the $I$-bundle in $H_{B}$, and let $\gamma_{2}$ be an essential arc in $Y$ disjoint from $\gamma_{1}$ with its endpoints on $y_{2}$. The existence of $\gamma_{2}$ is guaranteed by the assumption that $Y$ is an essential subsurface that is not an annulus or 3-holed sphere. The curve $b\in B$ is disjoint from $\gamma$, implying that $d(A,B)\leq 2$ as required.\\

It has now been shown that if $B$ also has an unbounded diameter in the subsurface projection to $Y$, there must be a $b\in B$ contained in $Y$ or $\Sigma\setminus Y$. Since $d(A,B)\geq 3$, this means that every element of $A$ must intersect every boundary component of $Y$. If $Y$ has more than one boundary component, the construction in the previous paragraph gives an element of $A$ that does not intersect every boundary component of $Y$. It follows that $Y$ can have only one boundary component. However, if $Y$ has only one boundary component, it must have genus at least one. The construction in the previous paragraph shows how to construct an element of $A$ intersecting $Y$ along an arc disjoint from any fixed curve in the interior of $Y$. If the curve $b$ is in $Y$, this gives a contradiction to the assumption that $d(A,B)\geq 3$. It follows that $b$ must be in $\Sigma\setminus Y$. However, this also gives a contradiction, because there is a curve in $Y$ disjoint from both $a$ and $b$.
\end{proof}

\subsection{Three subsets of a handlebody set}
\label{sec:largesubsurface}
This subsection describes how to obtain annular subsurfaces to which $A$ has large diameter in the subsurface projection. The construction will be used to partition $A$ into three subsets. Different techniques will then be used to prove Lemma \ref{maintheorem} for each of these subsets.\\

\textbf{The disjoint sum subset }$\mathbf{s_{1}}$\textbf{ of }$\mathbf{A}$\textbf{.} Recall the assumption that $A$ is distance 1 from $c$. Define a subset $s_{1}$ of $A$, the ``disjoint sum subset'', consisting of all the curves in $A$ that are disjoint from $c$, as well as their band sums. To see that $s_{1}$ does not have bounded diameter in the subsurface projection to $\mathcal{A}_{c}$, note that an arc in the definition of band sum of two curves disjoint from $c$ can be chosen to wrap around $c$ any number of times. However, for this reason, $s_{1}$ is invariant under the action of $D_{c}$.\\

\begin{rem}
\label{s1rem}
Although the subset $s_{1}\subset A$ is invariant under $D_{c}$, Theorem 1.11 of \cite{UO} (see Lemma \ref{ctwistinglemma2}) implies that $A\setminus s_{1}$ is not.
\end{rem}

The disjoint sum subset is not the only of the 3 subsets with unbounded diameter in the subsurface projection to $\mathcal{A}_{c}$, as illustrated in the next example. \\

\begin{ex}
\label{simpletwist}
Let $c_{1}$ be a curve in $A$ that intersects $c$, and $d_{1}$ be a curve in $A$ disjoint from $c$ and $c_{1}$. As illustrated in Figure \ref{increasetwisting} part (a), it is possible to take a band sum of $c_{1}$ with $d_{1}$ in such a way that a connected component of $c_{1}\cap \mathcal{A}_{c}$ twists once more around $c$ than previously. Call the resulting curve $c_{2}$. As shown in Figure \ref{increasetwisting} part (b), it is then possible to take a band sum of $c_{2}$ with a second copy of $d_{1}$, in such a way as to increase by one the number of times a connected component of $c_{2}\cap \mathcal{A}_{c}$ twists around $c$. Call the resulting curve $c_{2}$. A sequence of curves $\{c_{i}\}$ is obtained by iteration.\\

\begin{figure}[!thpb]
\centering
\includegraphics[width=0.7\textwidth]{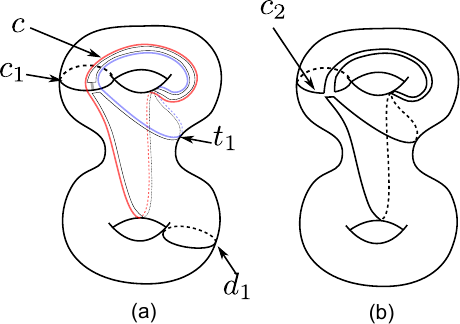}
\caption{Part (a) shows an example of curves $c_{1}$, $d_{1}$, $t_1$ and $c$. The curves $c$ is red, $t_{1}$ is blue and the thin line is $d_{1}$. The thick grey arc is the arc along which the band sum is performed. The curve $c_{2}$ is shown in part (b).}
\label{increasetwisting}
\end{figure}

By construction, the set $\{c_{i}\}$ has infinite diameter in the subsurface projection to $\mathcal{A}_{c}$, but it also has infinite diameter in the subsurface projection to another annulus. This annulus has core curve $t_{1}$ disjoint from both $c$ and $d_{1}$. 
\end{ex}

In Example \ref{simpletwist}, each band sum alters $c_{i}$ in such a way that it is twisted around  $c$, then backwards around a curve $t_{1}$ that cobounds an annulus with $c$ in $H_{A}$. Outside of the annuli $\mathcal{A}_{c}$ and $\mathcal{A}_{t_{1}}$, $c_{i}$ is unchanged. The letter $t$ (sometimes with a subscript) will be used to denote a simple curve in $\Sigma$, disjoint from and not homotopic to $c$ in $\Sigma$, such that $c\cup -t$ is the boundary of an annulus in $H_{A}$. The set of all such curves in $\Sigma$ will be denoted by $T$.\\

\begin{lem}[Theorem 1.11 of \cite{UO}] 
\label{ctwistinglemma2}
Suppose $\phi:H_{A}\rightarrow H_{A}$ is an automorphism with the property that $\phi|_{\Sigma}$ is a composition of Dehn twists around disjoint curves. Then $\phi$ is a composition of Dehn twists around the boundaries of a disjoint collection of properly embedded disks and annuli.
\end{lem}

An important special case of Lemma \ref{ctwistinglemma2} is that when $t$ and $c$ are homotopically nontrivial curves in $\Sigma$ that bound an embedded annulus $L(t,c)$ in the interior of $H_{A}$, the mapping class $D^{-1}_{t}D_{c}$ maps $A$ to itself.\\

\textbf{The $c$-bounded and $c$-unbounded subsets $\mathbf{s_{2}}$ and $\mathbf{s_{3}}$ of $\mathbf{A}$.}  Suppose $\{a_{1}, \ldots, a_{3g-3}\}$ is a choice of pairwise disjoint curves in $A$ as stated under ``Basic assumptions'' in Subsection \ref{sec:handlebody}. The $c$-bounded subset $s_{2}$ of $A\setminus s_{1}$ is defined to be the set of curves within radius $R$ of the curves $Z=\{a_{m}, \ldots, a_{3g-3}\}$ in the subsurface projection to $\mathcal{A}_{c}$, where $R$ should be thought of as a sufficiently large natural number whose size will be estimated later. The choice of $Z$ will only affect the size of $R$, not the existence of an $R$ with the desired properties. A specific choice for $Z$ will be given below. The $c$-unbounded subset $s_{3}$ is defined to be $A\setminus (s_{1}\cup s_{2})$.\\ 

The strategy is to show that $R$ can be chosen large enough such that $s_{2}$ contains every curve $s$ in $A\setminus s_{1}$ with the property that Theorem \ref{MM} allows the existence of a geodesic $\gamma(s,b)$ from $s$ to a point $b\in B$ such that every vertex of $\gamma(s,b)$ intersects every $t\in T$. Suppose $Z$ contains a curve $z$ in $A\setminus s_{1}$ with the smallest possible intersection number with $c$. Fix $b\in B$ and choose $z$ in such a way that 
\begin{equation}
\label{firstconstant}
K_{1}:=\max_{Y\in W_{T}}d_{Y}(z, b)
\end{equation}
is minimised, where $W_{T}$ is the subset of compact, connected, essential, embedded subsurfaces of $\Sigma\setminus c$, each of which contains a curve in $T$. The constant $K_{1}$ is finite because it is less than the length of a hierarchy path (as defined in \cite{MasurMinsky}) between two points.\\

The purpose of the next corollary is to understand how to construct curves in $s_{3}$.\\

\begin{figure}[!thpb]
\centering
\includegraphics[width=0.8\textwidth]{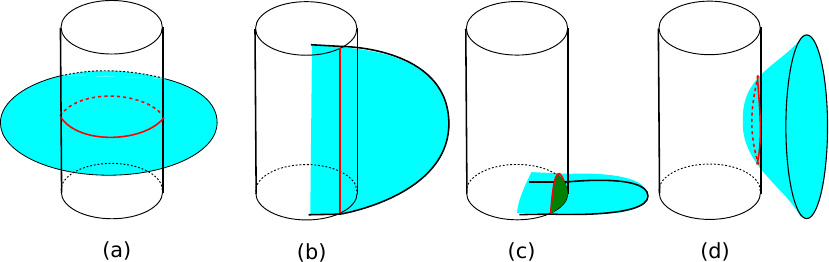}
\caption{The different ways an annulus $L(t,c)$ embedded in $H_{A}$ with boundary curves $t$ and $c$ in $\Sigma$ can intersect a disk $\mathcal{D}_{a}$ for $a$ in $A$. A portion of $\mathcal{D}_{a}$ is shown in blue. If $L(t,c)$ and $\mathcal{D}_{a}$ are embedded and homotoped to intersect transversely, connected components of the intersection are either homotopic to the core curve of $L(t,c)$ (a), arcs with one endpoint on $t$ and one on $c$ (b), arcs with both endpoints on the same boundary curve of $L(t,c)$ (c) or curves bounding a disk in both $L(t,c)$ and $\mathcal{D}_{a}$ (d).  In part (a) the curve shown in red bounds a disk in $\mathcal{D}_{a}$. This type of intersection is possible only if $c\in A$. There are intersections of the type shown in part (b) iff $a\in A\setminus s_{1}$. The intersection in part (d) can be removed by a homotopy of $\mathcal{D}_{a}$.}
\label{diskannulus}
\end{figure}

\begin{lem}
\label{25and27}
For any $t\in T$, every curve in $s_{3}$ can be obtained as $D^{k}_{c}D^{-k}_{t}(a)$ for some $a\in s_{2}$ and some $k\in \mathbb{Z}$. 
\end{lem}
\begin{proof}
Choose and fix $s\in s_{3}$ and $t\in T$. The integer $k(s)$ can then be chosen such that $d_{\mathcal{A}_{c}}(D^{-k}_{c}D^{k}_{t}(s), Z)\leq R$. By Lemma \ref{ctwistinglemma2}, $ D^{-k}_{c}D^{k}_{t}(s)\in A$. Since $D^{-k}_{c}D^{k}_{t}(s)$ is not in $s_{3}$, it remains to show that $D^{-k}_{c}D^{k}_{t}(s)$ is not in $s_{1}$. This follows from Remark \ref{s1rem} and the following construction showing that $s_{1}$ is also invariant under Dehn twists around $t$. \\

The remainder of this proof involves understanding the different ways a disk embedded in $H_{A}$ with boundary in $A$ can intersect an annulus $L(t,c)$ embedded in $H_{A}$ with boundary curves $t$ and $c$ in $\Sigma$. Examples of the different possibilities are illustrated in Figure \ref{diskannulus}.\\

A disk with boundary $e$ in $s_{1}$ is obtained as a band sum of curves disjoint from $c$. Cut the bands intersecting $c$ to obtain a set of curves $\{e_{1},\ldots, e_{i}\}$ in $s_{1}$ that are disjoint from $c$. This corresponds to performing bigon surgeries on $\mathcal{D}_{e}$, where an example of one such bigon is shown in green in Figure \ref{diskannulus} (c). It will now be shown that this gives a set of curves $\{e_{1},\ldots, e_{i}\}$ in $A$ that are all disjoint from both $c$ and $t$. Suppose otherwise, then there is an $e_{j}\in \{e_{1},\ldots, e_{i}\}$ that has nonzero geometric intersection number with $t$. The intersection of the disk $\mathcal{D}_{e_{j}}$ with $L(t,c)$ cannot contain a loop homotopic to a nonzero multiple of the core curve of  $L(t,c)$ as shown in part (a) of Figure \ref{diskannulus}, because then the Jordan curve theorem in $\mathcal{D}_{e_{j}}$ would imply that $c$ and $t$ are in $A$. Since $e_{j}$ does not intersect $c$, ignoring intersections of $\mathcal{D}_{e_{j}}$ with $L(t,c)$ that can be removed by a homotopy of $\mathcal{D}_{e_{j}}$, the only type of intersection between $\mathcal{D}_{e_{j}}$ and $L(t,c)$ is as shown in Figure \ref{diskannulus} (c). It follows that the intersection of $\mathcal{D}_{e_{j}}$ with $L(t,c)$ determine bigon surgeries that can be performed on $\mathcal{D}_{c_{j}}$ to obtain a set of disks with boundary disjoint from $L(t,c)$. Since $c_{j}$ has nonzero intersection number with $t$, the boundaries of these disks cannot all be homotopically trivial in $\Sigma$. This construction shows that elements of $s_{1}$ can be obtained as band sums of curves disjoint from both $c$ and $t$. Consequently, $s_{1}$ is invariant under $D_c$ and  $D_t$, for the same reason that $s_{1}$ is invariant under $D_{c}$, see for example the comment preceeding Remark \ref{s1rem}.
\end{proof}

Recall that $z\in Z$ is the curve in the definition of $s_{2}$ in Equation \eqref{firstconstant}.

\begin{lem}
\label{intersectionnumber2}
 Let $d$ be a curve in $s_{3}$ for which $d_{\mathcal{A}_{c}}(z, d)=|k|$ for $k\in \mathbb{Z}$, in particular, $|k|\geq R$. Suppose the sign of $k$ is chosen such that $i(z,D_{c}^{-k}(d))<i(z,d)$. It holds that $i(z,D_{c}^{-k}(d))\geq |k|-1$. Moreover there is a homotopy class of arcs $d^{a}$ of $d\cap (\Sigma\setminus c)$ and $z^{a}$ of $z\cap (\Sigma\setminus c)$ such that $d^{a}$ and $z^{a}$ both intersect every curve $t\in T$.
\end{lem}

\begin{proof}
%The surgery construction in the  proof of Lemma \ref{25and27} showed that every curve in $A\setminus s_{1}$, in particular the curve $z$, has nonzero geometric intersection number with every $t\in T$.\\

Since $|k|>1$, the assumption that $c$ is not in $A$ implies that the homology class $[D_{c}^{-k}(d)]-[z]$ in $H_{1}(H_{A};\mathbb{Z})$ is nonprimitive. The surgery construction in the proof of Lemma \ref{25and27} showed that every curve in $A\setminus s_{1}$, in particular $z$, must intersect every primitive homology class of which this is a multiple, i.e. there exist arcs of $\mathcal{D}_{z}\cap L(t,c)$ as shown in Figure \ref{diskannulus} (b). Since $[D_{c}^{-k}(d)]-[z]$ is at least a $|k|$ multiple of a primitive homology class, the bound $i(z,D_{c}^{-k}(d))\geq |k|-1$ follows immediately.\\

To prove the statement about the arcs $d^{a}$ and $z^{a}$, recall from the proof of Lemma \ref{25and27} that by homotoping the disk if necessary, the intersections of the disks $\mathcal{D}_{d}$ and $\mathcal{D}_{z}$ with $L(c,t_{i})$ can all be made to be of the type shown in Figure \ref{diskannulus} (b) or (c). This will be assumed from now on. Since $d$ and $z$ are not in $s_{1}$, it was also shown that not all intersections can be of the type shown in Figure \ref{diskannulus} (c). Then $d^{a}$ and $z^{a}$ can be chosen to be arcs containing one side of a bigon of the type shown in Figure  \ref{diskannulus} (b). For the curve $t\in T$ on the boundary of the annulus $L(t,c)$, it is clear that this arc intersects $t$. To show that such an arc intersects every curve in the set $T$, there are two cases to consider, as illustrated in Figure \ref{arcfigure}. In the first case, the existence of a $t'\in T$ for which $d^{a}$ does not intersect $t'$ implies that the intersections of $\mathcal{D}_{d}$ with $L(c,t')$ could all be taken to be of the type shown in Figure \ref{diskannulus} (c). This contradicts the assumption that $d\in (A\setminus s_{1})$. In the second case, if $d^{a}$ does not intersect $t'$, $d$ has a different algebraic intersection number with $t'$ as with $c$, contradicting Remark \ref{algebraicint} and Lemma \ref{ctwistinglemma2}.
\end{proof}

\begin{figure}[!thpb]
\centering
\includegraphics[width=0.5\textwidth]{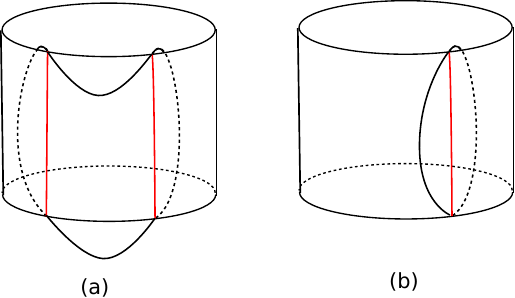}
\caption{The annulus $L(t,c)$ in $H_{A}$ has a tubular neighbourhood that can be oriented. This makes it possible to distinguish between bigons of the type shown in Figure \ref{diskannulus} (b) that are to the right of $L(t,c)$ and to the left of $L(t,c)$. Each connected component of $d\cap (\Sigma\setminus c)$ that contains one edge of such a bigon, contains an edge of two such bigons. These bigons are either both on the same side of $L(t,c)$ as shown in part (a) or both on opposite sides of $L(t,c)$ as shown in part (b).}
\label{arcfigure}
\end{figure}

\begin{rem}
\label{pairedrem}
The first part of Lemma \ref{intersectionnumber2} could alternatively have been proven by constructing a sequence $\gamma$ of multicurves in $A\setminus s_{1}$ representing a path from $z$ to $d$ in an appropriate analogue of $\mathcal{C}(\Sigma)$. Here the $i$th vertex of $\gamma$ is obtained from the $i-1$st vertex by a bigon surgery, using a bigon in a connected component of $\mathcal{D}_{d}\cap (H_{A}\setminus \mathcal{D}_{\gamma_{i}})$. As a result of the bigon surgery, $i(\gamma_{i}, d)\leq i(z,\gamma_{i-1})-2$. Lemma \ref{intersectionnumber2} implies that minimum number, $|k|-1$, of intersections between $z$ and $d$ resulting from the constraint $d_{\mathcal{A}_{c}}(z, d)=|k|$ are ``paired'' with at least the same number of intersections between $z$ and $D_{c}^{-k}(d)$. This inspires the next definition.
\end{rem}

\textbf{Paired intersections.} For $a_{1}$ and $a_{2}$ in $A$, two crossings of $a_{1}$ and $a_{2}$ will be said to be \textit{paired} if the two crossings are the endpoints of an embedded arc in the intersection of $\mathcal{D}_{a_{1}}$ with $\mathcal{D}_{a_{2}}$. Similarly, for $a_{1}\in a$ and $t\in T$, two crossings of $a_{1}$ with the multicurve $t\cup c$ will be called \textit{paired} if the crossings are the endpoints of an arc in the intersection of $\mathcal{D}_{a_{1}}$ with $L(t,c)$. A crossing of $a_{1}$ with $t$ will be said to be \textit{paired with a crossing in }$\mathcal{A}_{c}$ if it is paired with a crossing on the boundary curve $c$ of $L(t,c)$.

\begin{rem}
\label{intrem}
For $d\in A\setminus s_{1}$, the intersection of $\mathcal{D}_{d}$ with $L(t,c)$ gives at least two bigons in $\mathcal{D}_{d}\cap (H_{A}\setminus L(t,c))$. If these two bigons are on opposite sides of $L(t,c)$ as in Figure \ref{arcfigure} (b), a curve $z\in A\setminus s_{1}$ with $i(c,z)=1$ can be constructed by attaching the two bigons to form $\mathcal{D}_{z}$. If the two bigons are on the same side of $L(t,c)$, a curve $z\in A\setminus s_{1}$ with $i(c,z)=2$ can be constructed by attaching the two bigons to a rectangle in $H_{A}$ with a pair of opposite sides on $L(t,c)$, as shown in Figure \ref{arcfigure} (a).
\end{rem}
% Ignore connected components of $d\cap (S\setminus \{c\})$ and $z\cap (S\setminus \{c\})$ that comprise one side of a bigon, as shown in Figure \ref{diskannulus} c). 
%\end{proof}
%When $\tilde{i}(d,c)=i(d,c)=1$ and $\tilde{i}(z,c)=i(z,c)=1$ the arc $d^{1}$ of $d\cap \Sigma\setminus \{c\}$ is the single connected component of $d\cap \Sigma\setminus \{c\}$, and the arc $z^{1}$ is the single connected component of $z\cap \Sigma\setminus \{c\}$. In the case that $\tilde{i}(c,d)<i(c,d)$ or $\tilde{i}(c,z)<i(c,z)$ the by now familiar construction with bigon surgeries shows that $d$ (respectively $z$) can be expressed as a band sum of $d'$ (respectively $z'$) with curves in $s_{1}$ disjoint from $c$. Consequently, $\tilde{i}(d',c)=i(d',c)$ and $\tilde{i}(z',c)=i(z',c)$, and the previous argument gives $i(z',D_{c}^{-k}(d'))\geq (k-2)\tilde{i}(d',c)\tilde{i}(z',c)$. Now $d'\cap \Sigma\setminus \{c\}$ has $\tilde{i}(d',c)$ connected components, each of which  eacand $z'\cap \Sigma\setminus \{c\}$ has $\tilde{i}(z',c)$ connected components, it is possible to chose one connected
%\end{proof}

It will now be shown that, within the $c$-unbounded subset $s_{3}$, moving out to infinity in the subsurface projection to $\mathcal{A}_{c}$ necessarily also involves moving out to infinity in some other subsurface. This will be used to obtain lower bounds on $d(D_{c}^{n}(A), B)$ for large $|n|$. The next theorem will be used to relate the intersection number from Lemma \ref{intersectionnumber2} to a large distance in a subsurface projection. Some notion is needed: $c(\Sigma)$ is the set of curves on $\Sigma$, $W$ is the set of essential, connected, nonannular subsurfaces of $\Sigma$ that are not thrice holed spheres and $[n]_{k}$ is the cutoff funtion defined as follows
\begin{equation*}
[n]_{k}=
\begin{cases}
n\text{ if }n> k,\\
0\text{ otherwise}
\end{cases}
\end{equation*}

\begin{lem}[Theorem 1.5 of \cite{Wat}]
\label{schleimerlem}
For every two vertices $v_{1}$ and $v_{2}$ of $\mathcal{C}(\Sigma)$ and all $k>0$, we have 
\begin{equation*}
\log_{2} i(v_{1},v_{2})\leq V_{g}(k)\left(1+\sum_{Y\subset W}[d_{Y}(v_{1},v_{2})]_{k}+\sum_{\mathcal{A}_{l}\text{, }l\subset c(\Sigma)}\log_{2}[d_{\mathcal{A}_{l}}(v_{1},v_{2})]_{k}\right) 
\end{equation*}
where $V_{g}(k)=\left(40000(2g-2)(k+200(3g-3)\right)^{3g-1}$, and $\log_{2}(0):=0$.
%For $\sigma\in \mathbb{N}$, an upper bound of $\sigma$ on the distance between two curves $e$ and $f$ in every subsurface projection gives a uniform upper bound on $i(e,f)$, depending only on $\sigma$ and $\Sigma$. This lemma also applies when $\Sigma$ is replaced by an essential embedded subsurface, and the curves $e$ and $f$ replaced by homotopy classes of arcs obtained by taking intersections of $e$ and $f$ with the subsurface. The homotopies in this case are assumed to keep the endpoints of the arcs on the boundary of the subsurface.
\end{lem}

\begin{lem}[Lemma 11 of \cite{Me2}]
\label{numberofarcs}
Suppose $c_{1}$ and $c_{2}$ are two curves in $\Sigma$ in minimal position, for example, $c_{1}$ and $c_{2}$ are geodesics. Then the number of isotopy classes of arcs $c_{1}\cap (\Sigma\setminus c_{2})$ is uniformly bounded from above by $3g-3$.
\end{lem}

Recall that $\beta(\Sigma)$ is the constant from Theorem \ref{MM}, and $z$ is the curve in the definition of $s_{2}$, given in Equation \eqref{firstconstant}.

\begin{lem}
\label{lem:largeY}
%For any $e$ in $s_{2}$ and any $f$ in $s_{3}$ with $d_{\mathcal{A}_{c}}(e, f)=k$ sufficiently large, there is a subsurface $Y(e, f)$ in $\Sigma\setminus c$ containing a curve $t(e,f)$ in $T$ for which $d_{Y}(e,f)\geq \beta(\Sigma)$. 
Choosing $R$ in the definition of $s_{2}$ sufficiently large ensures that for every $d$ in $s_{3}$ there is a subsurface $Y(d)$ containing a curve $t(d)\in T$ for which $d_{Y(d)}(d,b)\geq \beta(\Sigma)$ for every $b\in B$. Here ``sufficiently large'' depends only on $\Sigma$.
\end{lem}

\begin{proof}
Recall that the surgery construction in the proof of Lemma \ref{25and27} showed that for every $t\in T$ there is a curve in $s_{1}$ disjoint from $t$. The assumption that $d(A,b)\geq 3$ then implies that every $b\in B$ intersects every $t\in T$. If a subsurface of $\Sigma$ contains a $t\in T$, it follows that every curve in $B$ is either contained in this subsurface or has essential intersection with the subsurface. Choosing $R$ in the definition of $s_{3}$ sufficiently large will now be shown to imply the existence of a $Y(d)$ containing a curve $t\in T$ with $d_{Y}(d,z)$ sufficiently large that the lemma then follows from Lemma \ref{distance3}.\\

Recall that  $K_{1}$ is the constant from Equation \eqref{firstconstant} and $\delta$ is the constant from Theorem \ref{masurschleimer}. The lemma will first be proven for any curve $d\in s_{3}$ with $i(d,c)=i(z,c)$ and $\hat{i}(d,c)=\hat{i}(z,c)$, where by Remark \ref{intrem}, $i(z,c)=1$ or in the case that $i(z,c)=2$, $\hat{i}(z,c)=0$. In the special case that $d=D_{c}^{k}D_{t}^{-k}(z)$, note that choosing $R$ larger than $\beta(\Sigma)+K_{1}+\delta$ will ensure that $d_{\mathcal{A}_{t}}(d,b)\geq \beta$ for all $b\in B$.\\

Informally, the notation $I(z,d)$ will mean the  ``smallest'' embedded, essential subsurface of $\Sigma\setminus c$ containing the arcs $z\cap (\Sigma\setminus c)$ and $d\cap (\Sigma\setminus c)$. The precise definition of $I(z,d)$ is given next. Let $\{z_{1}, \ldots, z_{k}\}$ be the isotopy classes of arcs of $z\cap (\Sigma\setminus c)$. 

These isotopy classes determine a set of pairwise disjoint rectangles, $\{R_{1}, \ldots, R_{k}\}$ contained in $\Sigma \setminus c$; one for each isotopy class. The rectangle $R_{i}$ corresponding to the $i$th isotopy class $z_{i}$ has a pair of opposite sides on the boundary of the subsurface $\Sigma \setminus c$, and the other pair of opposite sides given by arcs in the same isotopy class as $z_{i}$. Denote by $\Sigma_{c,k}$ the subsurface of $\Sigma\setminus c$ obtained by removing all the rectangles $\{R_{1}, \ldots, R_{k}\}$. In the subsurface $\Sigma_{c,k}$, next construct a set of rectangles $\{R'_{1'}, \ldots, R'_{k'}\}$ corresponding to the isotopy classes of arcs of $d\cap \Sigma_{c,k}$, analogous to the construction of $\{R_{1}, \ldots, R_{k}\}$. Glue the rectangles in the sets $\{R_{1}, \ldots, R_{k}\}$ and $\{R'_{1}, \ldots, R'_{k'}\}$ along common boundary points to obtain an embedded subsurface. For each boundary component of this embedded subsurface that bounds a polygon in $\Sigma\setminus c$, glue the polygon along the boundary. The resulting embedded subsurface is $I(z,d)$. \\

%\begin{figure}[!thpb]
%\centering
%\includegraphics[width=0.4\textwidth]{cloop.pdf}
%\caption{The intersection of $I(z,d)$ (shown in green) with an annular subsurface of $\Sigma$ with core curve $c$. The curve $z$ is shown in red and $d$ in black, and the arcs $a_{1}$ and $a_{2}$ are blue.}
%\label{cloop}
%\end{figure}

The difference $[z]-[D_{c}^{-k}(d)]$ is a nonprimitive homology class in $H_{1}(H_{A};\mathbb{Z})$ given by $\lambda[t]$, where $|\lambda|\geq R$. It follows that, not only does $I(z,d)$ contain a curve $t\in T$, but any disjoint union $m$ of curves and properly embedded arcs in $I(z,d)$ that intersects every curve in $T$ contained in $I(z,d)$ satisfies $i(z,m)+i(d,m)\geq R$. \\

The existence of a large subsurface projection to a subsurface $Y(d)$ of $I(z,d)$ is guaranteed by Lemmas \ref{intersectionnumber2}, \ref{schleimerlem} and \ref{numberofarcs} whenever $R$ is chosen sufficiently large. An algorithm for constructing a subsurface $Y(d)$ will now be derived.\\

If the distance between $z$ and $d$ in the subsurface projection to $I(z,d)$ is greater than or equal to $\beta +K_{1}+\delta$, set $Y(d)$ equal to $I(z,d)$. Otherwise, a subsurface $I(z,d)_{1}$ will be constructed. The construction depends on whether an arc $a_{1}$ with properties given below exists.\\

Case 1: there exists an essential, properly embedded arc $a_{1}$ in $I(z,d)$ that has small intersection numbers with $d$ and $z$. In this case, cut $I(z,d)$ along $a_{1}$ to obtain the subsurface $I(z,d)_{1}^{'}$. The notion of ``small'' intersection numbers will be defined later; for the moment it is enough to observe that an arc $a_{1}$ with sufficiently small intersection numbers with $d$ and $z$ will give a boundary component $c_{s}$ of a tubular neighbourhood of $a_{1}\cup\partial I(z,d)$ with $d_{I(z,d)}(d,c_{s})+d_{I(z,d)}(c_{s},z)\leq \beta+K_{1}+\delta$, showing that $I(z,d)$ could not be $Y(d)$.\\

If the subsurface $I(z,d)_{1}^{'}$ has any connected components consisting of $n$-holed spheres, replace each $n$-holed sphere in $I(z,d)_{1}^{'}$ by the set of annuli with core curves in the $n$-holed sphere, to which the subsurface projections of $d$ and $z$ have distance at least $\beta+K_{1}+\delta$. Also, if any connected components of $I(z,d)_{1}^{'}$ have boundary curves that define annuli to which the subsurface projections of $z$ and $d$ have distance at least $\beta+K_{1}+\delta$, add these annuli to the set of subsurfaces. Call the resulting set of subsurfaces $I(z,d)_{1}$. \\

Case 2: when there does not exist an arc $a_{1}$ as defined above, since the distance between $z$ and $d$ in the subsurface projection to $I(z,d)$ is assumed to be less than $\beta +K_{1}+\delta$, there exists a nonperipheral curve $a_{1}$ in $I(z,d)$ with small intersection numbers with $d$ and $z$. In this case, $I(z,d)_{1}^{'}$ is obtained by cutting $I(z,d)$ along $a_{1}$. The set of subsurfaces $I(z,d)_{1}$ is then obtained from $I(z,d)_{1}^{'}$ in the same way as in case 1. \\

Either every element of $I(z,d)_{1}$ has the property that the distance between $z$ and $d$ in the subsurface projection to the embedded subsurface corresponding to that element is bounded from below by $\beta(\Sigma)+K_{1}+\delta$, or the construction can be repeated on the elements not satisfying this condition to obtain $I(z,d)_{2}$. It is possible to iterate the construction $j$ times to obtain $I(z,d)_{j}$, where by Lemma \ref{numberofarcs}, $j<3g-3$. At every step of the iteration, the arc or curve $a_{i}$, $i=1, \ldots, j-1$ is chosen such that 
\begin{equation}
\label{restriction}
i(a_{i}, z)+i(a_{i},d)\leq\frac{R}{6g-6}.
\end{equation}
The justification for Inequality \eqref{restriction} is the following: Suppose there is an $i$ with $1\leq i\leq j-1$ for which \textit{every} possible choice of curve or arc $a_{i}$ violated Inequality \eqref{restriction}. In this case, for sufficiently large $R$, Lemma \ref{schleimerlem} would give that the distance between $z$ and $d$ in the subsurface projection to every element of $I(z,d)_{i-1}$ is at least $\beta(\Sigma)+K_{1}+\delta$. That being the case, set $j=i-1$ and terminate the construction.\\

 It remains to be shown that a subsurface in $I(z,d)_{j}$ contains a curve in $T$. When the choices $a_{1}, \ldots, a_{i-1}$ are all made in accordance with Inequality \eqref{restriction}, the curves and arcs $\{a_{1}, \ldots, a_{i-1}\}$ cannot have large enough intersection numbers in total with $d$ and $z$ to cut though every curve in $T$ contained in $I(z,d)$. It is therefore possible to find an element of $I(z,d)_{j}$ containing a curve in $T$. This subsurface is then called $Y(d)$. This concludes the proof of the lemma in the case that $i(d,c)=i(z,c)$.\\

To prove the lemma when $i(d,c)>i(z,c)$, note that the arcs $z^{a}$ and $d^{a}$ guaranteed by Lemma \ref{intersectionnumber2} are contained in the subsurface $I(z,d)_{1}$ and must have intersection number at least $R-1$ in this subsurface. The rest of the argument is then identical to the proof in the case that $i(d,c)=i(z,c)$.
\end{proof}

It is now possible to put all the pieces together to give a proof of Lemma \ref{maintheorem}.\\

\begin{lem}[Lemma \ref{maintheorem} from Subsection \ref{sec:Hempel}]
Suppose that $d(A,B)\geq 3$, and let $c$ be a curve, which is no longer assumed to be distance one from $A$. It follows that for all but finitely many $n\in \mathbb{Z}$, 
\begin{equation*}
d(D_{c}^{n}(A), B)\geq d(A,B)
\end{equation*}
\end{lem}

\begin{proof}
If $A$ and $B$ could each be shown to have finite diameter in the subsurface projection to $\mathcal{A}_{c}$, the theorem would follow from the same argument as for Proposition \ref{twocurves}. By Theorem \ref{masurschleimer}, this is the case for example when $d(c, A)\geq 2$ and $d(c, B)\geq 2$. This special case of the lemma was also proven in Theorem 1.4 of \cite{Yoshizawa}. Alternatively, if $D_{c}$ takes $A$ or $B$ to itself, e.g. $c\in B$ or $c\in A$, the lemma also follows. It remains to prove the lemma under the assumption that $d(c, A)=1$, $d(c,B)\geq 2$ and $d(A,B)\geq 3$; the argument in the case that $d(c, B)=1$, $d(c,A)\geq 2$ is symmetric.\\

The remaining case of the lemma will be proven by showing the following three statements. The first statement is that $d(D_{c}^{n}(s_{1}), B)= d(s_{1},B)$ for every $n\in \mathbb{Z}$. The second statement is that for all but finitely many $n\in \mathbb{Z}$, $d(D_{c}^{n}(s_{2}), B)\geq d(s_{2},B)$. The third statement is that for all but finitely many $n\in \mathbb{Z}$, $d(D_{c}^{n}(s_{3}), B)\geq d(s_{2},B)$.\\

The first statement is a consequence of the first part of Remark \ref{s1rem}. By Lemma \ref{distance3} the second statement follows from an identical argument to that given in the proof of Proposition \ref{twocurves}, because both $s_{2}$ and $B$ have bounded diameters in the subsurface projection to $\mathcal{A}_{c}$. The third statement follows from a double iteration of the construction in the proof of Proposition \ref{twocurves}, as will now be explained.\\

\begin{figure}[!thpb]
\centering
\includegraphics[width=\textwidth]{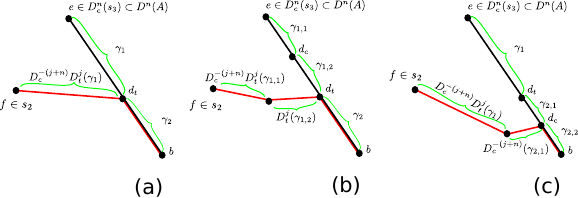}
\caption{For sufficiently large $n$, a geodesic $\gamma$ from $e\in D_{c}^{n}(s_{3})$ to $b\in B$ determines a path of the same length from $f=D_{c}^{-(j+n)}D_{t}^{j}(e)\in s_{2}$ to $b\in B$. Part (a) shows the construction of the path when the vertex $d_{t}$ is disjoint from $c$ as well as $t$, part (b) shows the construction when there is a vertex $d_{c}$ disjoint from $c$ in the segment $\gamma_{1}$ of $\gamma$, and part (c) shows the construction when there is a vertex $d_{c}$ disjoint from $c$ in the segment $\gamma_{2}$ of $\gamma$.}
\label{finalfigure}
\end{figure}

Denote by $e$ a curve in $D_{c}^{n}(s_{3})$, where $|n|$ is large enough to ensure that $d_{\mathcal{A}_{c}}(e, B)>\beta$. Theorem \ref{MM} together with Lemma \ref{lem:largeY} ensure that a geodesic $\gamma$ from $e$ to $b\in B$ must pass through a vertex $d_{t}$ disjoint from a curve $t$ in $T$, and a vertex $d_{c}$ disjoint from $c$. Choose $j\in \mathbb{Z}$ such that $f=D_{c}^{-j}D_{t}^{j}D_{c}^{-n}(e)\in s_{2}$ as guaranteed by Lemma \ref{25and27}. In the case that $d_{t}$ is also disjoint from $c$, a path from $f$ to $b$ of length equal to the length of $\gamma$ is constructed as in Figure \ref{finalfigure} (a). When $d_{t}$ intersects $c$, the vertex $d_{t}$ will not be fixed by $D_{c}^{-(j+n)}D_{t}^{j}$. In this case, either the subpath $\gamma_{1}$ or $\gamma_{2}$ contains the vertex $d_{c}$. In the first case, a path from $f$ to $b$ with the same length as $\gamma$ is constructed as in Figure \ref{finalfigure} (b), and in the second case, a path from $f$ to $b$ with the same length as $\gamma$ is constructed as in Figure \ref{finalfigure} (c). This concludes the proof of the third statement, and hence of the proof.
\end{proof}

%It follows from the constraints on $R$ discussed above that for any $n$, a geodesic  shorter than $d(A,B)$ from a vertex representing a curve $e$ in $D_{c}^{n}(s_{3})$ to a vertex representing a curve $b$ in $B$ must pass through a vertex representing a curve disjoint from some $t$. Given a geodesic $\gamma$ from $e$ to $b$, this will be used to construct a path $\gamma^{'}$ from a vertex $f:=D_{c}^{-j}D_{t}^{j}D_{c}^{-n}(e)$ in $A$ to the vertex $b$, that is no longer than $\gamma$. Here $j$ is chosen such that $D_{\mathcal{A}_{c}}(f,b)\leq 1$.\\

%Denote by $d_{t}$ the vertex on $\gamma$ representing a curve disjoint from $t$. The path $\gamma^{'}$ is composed of two geodesic segments. One geodesic segment connects $f$ to a vertex $t^{'}$, and the other connects $t^{'}$ to $b$.  If $d_{t}$ is disjoint from $c$, then $t^{'}$ and $d_{t}$ coincide. Otherwise, $t^{'}$ is obtained from $d_{t}$ by performing some number of Dehn twists around $c$, so as to minimise $d_{\mathcal{A}_{c}}(t^{'}, f)$ and hence also $d_{\mathcal{A}_{c}}(t^{'}, b)$. It follows that $d(f,t^{'})\leq d(e,d_{t})$ and $d(d_{t},b)\leq d(t^{'},b)$, and the length of $\gamma^{'}$ is less than or equal to that of $\gamma$, as required.
%\end{proof}

\bibliographystyle{plain}
%\nocite{*}
\bibliography{mcgbib}

\end{document}